%% file: preprint_paper_ppa_v4.tex
\documentclass[12pt]{article}
\usepackage{fullpage}

\usepackage{stmaryrd}
\usepackage{latexsym}     
\usepackage{amsmath}      
\usepackage{amsfonts}     
\usepackage{amssymb}      
\usepackage{euscript}
\usepackage{enumitem}
\usepackage{algpseudocode}
\usepackage{algorithm}

\usepackage[utf8]{inputenc}
\usepackage[T1]{fontenc}

\usepackage{tikz}
\usepackage{fancybox}
\usepackage{xcolor}
\usepackage{makeidx}
\usepackage{todonotes}

\usepackage{mathtools}

\usepackage[normalem]{ulem}

\usepackage[english]{babel}

\definecolor{bluegreen}{rgb}{0.0, 0.87, 0.87}

\input{commandes}

\newcommand{\Jad}{J\add}
\newcommand{\jad}{j\add}
\newcommand{\jc}{j\coupl}
\newcommand{\Jc}{J\coupl}
\newcommand{\J}{J}
\newcommand{\ban}{\espacea{U}}
\newcommand{\messpace}{\espacea{W}}
\newcommand{\optaux}{\widetilde{\va{U}}}
\newcommand{\jconv}[1]{\expec{\J\vardelim{\va{U}_{#1}} - J\vardelim{u\opt}}}
\newcommand{\jconvany}[2]{\expec{\J\vardelim{\va{U}_{#1}} - J\vardelim{\va{U}_{#2}}}}
\newcommand{\hkintegrand}{\tau_{k}}

\usepackage[pdftex,colorlinks=true,linkcolor=blue]{hyperref}

\newtheorem{theorem}{Theorem}[section]
\newtheorem{lemma}[theorem]{Lemma}
\newtheorem{corollary}[theorem]{Corollary}
\newtheorem{proposition}[theorem]{Proposition}

\newtheorem{definition}[theorem]{Definition}

\newtheorem{example}[theorem]{Example}

\newtheorem{remark}[theorem]{Remark}

\title{The stochastic Auxiliary Problem Principle in Banach spaces: measurability and convergence}
\author{
  Thomas Bittar, Pierre Carpentier, Jean-Philippe Chancelier, J\'{e}r\^{o}me Lonchampt
}

\begin{document}

\maketitle

\begin{abstract}
  The stochastic Auxiliary Problem Principle (APP) algorithm is a general
  Stochastic Approximation (SA) scheme that turns the resolution of an original
  {convex} optimization problem into the iterative resolution of a
  sequence of auxiliary problems. This framework has been introduced to design
  decomposition-coordination schemes but also encompasses many well-known SA
  algorithms such as stochastic gradient descent or stochastic mirror
  descent. We study the stochastic APP in the case where the iterates lie in a
  Banach space and we consider an additive error on the computation of the
  subgradient of the objective. In order to derive convergence results or
  efficiency estimates for a SA scheme, the iterates must be random
  variables. This is why we prove the measurability of the iterates of the
  stochastic APP algorithm. Then, we extend convergence results from the Hilbert
  space case to the reflexive {separable} Banach space case. Finally, we derive
  efficiency estimates for the function values taken at the averaged sequence of
  iterates or at the last iterate, the latter being obtained by adapting the
  concept of modified Fej\'{e}r monotonicity to our framework.
\end{abstract}


\section{Introduction}

Let $\ban$ be a {reflexive separable} Banach space whose norm is denoted by $\norm{\cdot}$, $(\omeg, \trib, \prbt)$
be a probability space and $(\messpace, \borel{\messpace})$ be a measurable
topological vector space with $\borel{\messpace}$ being the Borel $\sigma$-field
on $\messpace$. We refer
to~\cite{bauschke_convex_2011,billingsley_probability_1995} for the definitions
of basic concepts in analysis and probability theory. We consider a stochastic
optimization problem of the form:
\begin{align}
  \min_{u \in \Uad} \ba{\J(u) := \Jc(u) + \Jad(u)} \text{ where }
  \left\{
  \begin{aligned}
    \Jc(u) &= \expec{\jc(u, \va{W})} \eqfinv\\
    \Jad(u) &= \expec{\jad(u, \va{W})} \eqfinp
  \end{aligned}
  \right.
  \label{eq:master_pb}
\end{align}
where $\Uad \subset \ban$ is a non-empty closed convex set,
$\va{W} : \omeg \rightarrow \messpace$ is a random variable,
{$\jc: \ban \times \messpace \rightarrow \bbR$ and
  $\jad: \ban \times \messpace \rightarrow \bbR$ are such that} $\jc(\cdot, w)$
and $\jad(\cdot, w)$ are proper, convex, {and} lower-semicontinuous
(\lsc) real-valued functions for all $w \in \messpace$.

Stochastic Approximation (SA) algorithms are the workhorse for solving
Problem~\eqref{eq:master_pb}. The SA technique has been originally introduced
in~\cite{kiefer_stochastic_1952,robbins_stochastic_1951} as an iterative method
to find the root of a monotone function which is known only through noisy
estimates. SA algorithms have been the subject of many theoretical
studies~\cite{bach_non-asymptotic_2011,karimi_non-asymptotic_2019,nemirovski_robust_2009,polyak_acceleration_1992}
and have applications in various disciplines such as machine learning, signal
processing or stochastic optimal
control~\cite{benveniste_adaptive_2012,kushner_stochastic_1997}. Back in 1990,
with decomposition applications in mind, Culioli and
Cohen~\cite{culioli_decomposition/coordination_1990} proposed a general SA
scheme in an infinite dimensional Hilbert space based on the so-called Auxiliary
Problem Principle (APP), called the stochastic APP algorithm. This algorithm
also encompasses several well-known algorithms such as stochastic gradient
descent, the stochastic proximal gradient algorithm or stochastic mirror
descent. Recently,~\cite{geiersbach_projected_2019,martin_multilevel_2019} apply
SA methods to solve PDE-constrained optimization problems. In this paper, we
extend the stochastic APP algorithm to the Banach case.

A SA algorithm is defined by a recursive stochastic update rule.  For
$k \in \bbN$, the $k$-th iterate of a SA algorithm is a mapping
$\va{U}_{k}: \omeg \rightarrow \ban$, where the range of $\va{U}_{k}$ is
included in $\Uad$.  We denote by $\proscal{\cdot}{\cdot}$ the duality pairing
between $\ban$ and its topological dual space $\ban\dual$. In the case where
$\jc$ is differentiable {with respect to $u$}, the $k$-th iteration of
the stochastic APP algorithm computes a minimizer $u_{k+1}$ such that:
\begin{align}
  u_{k+1} \in \argmin_{u \in \Uad} K(u) + \proscal{\varepsilon_{k}\nabla_{u} \jc(u_{k}, w_{k+1}) - \nabla K (u_{k})}{u} + \varepsilon_{k} \jad(u, w_{k+1}) \eqfinv
  \label{eq:stoch-ppa}
\end{align}
where $\varepsilon_{k} > 0$ is a positive real, $w_{k+1}$ is a realization of
the random variable $\va{W}$ and $K$ is a user-defined
Gateaux-differentiable\footnote{{We use~\cite[Definition
    2.43]{bauschke_convex_2011} for the Gateaux-differentiability, which
    requires the linearity and boundedness of the directional derivative.}}
convex function. The role of the function $K$ is made clear in
Section~\ref{sec:desc_ex}. In the context of the APP,
Problem~\eqref{eq:stoch-ppa} is called the \emph{auxiliary problem} and the
function $K$ is called the \emph{auxiliary function}.  Let us now briefly expose
how this scheme reduces to well-known algorithms for particular values of $K$
and $\jad$.

The most basic SA scheme is stochastic gradient descent. Assume that $\ban$ is a
Hilbert space, $\Uad = \ban$ and $\jad = 0$. The $k$-th iteration is given by:
\begin{align}
  u_{k+1} = u_{k} - \varepsilon_{k} \nabla_{u} \jc(u_{k}, w_{k+1}) \eqfinp
  \label{eq:sgd}
\end{align}
This is exactly the stochastic APP algorithm~\eqref{eq:stoch-ppa} with $\jad= 0$ and $K = \frac{1}{2}\sqnorm{\cdot}$ where $\norm{\cdot}$ is the norm induced by the inner product in $\ban$.

In the case where $\jc$ is differentiable and $\jad$ is non-smooth but with a
proximal operator that is easy to compute, proximal
methods~\cite{atchade_perturbed_2017,parikh_proximal_2014} are particularly
efficient, even in a high-dimensional Hilbert space $\ban$. An iteration of the
stochastic proximal gradient algorithm is:
\begin{align}
  u_{k+1} \in \argmin_{u \in \ban} \frac{1}{2\varepsilon_{k}}\sqnorm{u_{k} - u} + \proscal{\nabla_{u} \jc(u_{k}, w_{k+1})}{u - u_{k}} + \jad(u, w_{k+1}) \eqfinp
  \label{eq:prox}
\end{align}
This is again the stochastic APP algorithm with $K = \frac{1}{2}\sqnorm{\cdot}$
but with a non zero function $\jad$. The proximal term
$\frac{1}{2\varepsilon_{k}}\sqnorm{u_{k} - u}$ forces the next iterate $u_{k+1}$
to be close to $u_{k}$ with respect to the norm $\norm{\cdot}$. When $\jad$ is
the indicator of a convex set, the stochastic proximal gradient method reduces
to stochastic projected gradient descent and when $\jad = 0$, this is just the
regular stochastic gradient descent~\eqref{eq:sgd}. Proximal methods are
well-suited for regularized regression problems in machine learning for example.

When $\ban$ is only a Banach space and not a Hilbert space,
Equation~\eqref{eq:sgd} does not make sense as $u_{k}\in \ban$ while
$\nabla_{u} \jc(u_{k}, w_{k+1}) \in \ban\dual$, thus the minus operation is not
defined. This difficulty is addressed with the mirror descent
algorithm~\cite{nemirovski_problem_1983}. The original insight of the method is
to map the iterate $u_{k}$ to $\nabla K(u_{k}) \in \ban\dual$, where $K$ is a
Gateaux-differentiable user-defined function. Then, we do a gradient step in
$\ban\dual$ and we map back the resulting point to the primal space $\ban$. The
function $K$ is called the \emph{mirror map} in this
setting~\cite{bubeck_convex_2015}. There is also a proximal interpretation of
mirror descent: instead of defining proximity with the norm $\norm{\cdot}$, the
mirror descent algorithm and its stochastic
counterpart~\cite{nemirovski_robust_2009} use a Bregman
divergence~\cite{bregman_relaxation_1967} that captures the geometric properties
of the problem:
\begin{align}
  u_{k+1} \in \argmin_{u \in \Uad} \frac{1}{\varepsilon_{k}}D_{K}(u, u_{k}) + \proscal{\nabla_{u} \jc(u_{k}, w_{k+1})}{u - u_{k}} \eqfinv
  \label{eq:mirror}
\end{align}
where $D_{K}$ is the Bregman divergence associated with
$K$:
\begin{align*}
  D_{K}(u, u') = K(u) - K(u') - \proscal{\nabla K(u')}{u-u'}, \quad u,u' \in \ban\eqfinp
\end{align*}
The function $K$ is sometimes called the \emph{distance-generating function} as
it defines the proximity between $u$ and $u'$. With
$K = \frac{1}{2}\sqnorm{\cdot}$, we get back to the setting of stochastic
gradient descent. The mirror descent algorithm is particularly suited to the
case where $\nabla_{u} \jc$ has a Lipschitz constant which is large with respect
to the norm $\norm{\cdot}$ but small with respect to some other norm that is
better suited to the geometry of the problem~\cite{nemirovski_robust_2009}. For
example, in the finite-dimensional case, the performance of stochastic gradient
descent depends on the Lipschitz constant of $\nabla_{u} \jc$ in the Euclidean
geometry. Hence, if the problem exhibits a non-Euclidean geometric structure,
stochastic mirror descent may be more efficient. Stochastic mirror descent
corresponds to the stochastic APP with a general function $K$ and $\jad = 0$.

The stochastic APP algorithm combines the ideas of mirror descent and of the
proximal gradient method. The iteration defined by~\eqref{eq:stoch-ppa} can be
equivalently written as:
\begin{align*}
  u_{k+1} \in \argmin_{u \in \Uad} \frac{1}{\varepsilon_{k}}D_{K}(u, u_{k}) + \proscal{\nabla_{u} \jc(u_{k}, w_{k+1})}{u - u_{k}} + \jad(u, w_{k+1})\eqfinp
\end{align*}
In the sequel, we stick to the formulation~\eqref{eq:stoch-ppa} and we consider
a more general version as $\jc$ is only assumed to be subdifferentiable and we
allow for an additive error on the subgradient
$\partial_{u} \jc(u_{k}, w_{k+1})$.  Figure~\ref{fig:link_algo} summarizes the
relationship between the four stochastic approximation algorithms that we have
introduced.

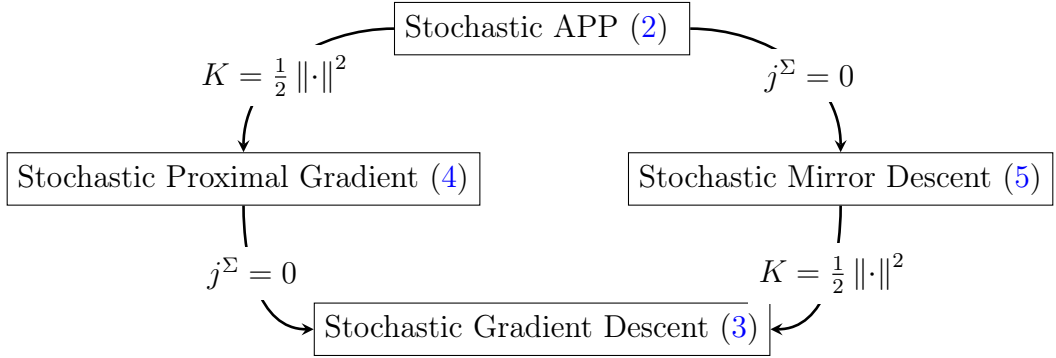
\begin{figure}[htbp]
  \centering
  \begin{tikzpicture}
    \begin{scope}[every node/.style={rectangle, draw, font=\normalsize, align=center, fill opacity=0.5, text opacity=1, draw opacity=1}]
      \node (APP) {Stochastic APP~\eqref{eq:stoch-ppa}
      };
      \node (prox) [xshift=-2cm, yshift=-2cm] at (APP.west) {Stochastic Proximal Gradient~\eqref{eq:prox}};
      \node (mirror) [xshift=2cm, yshift=-2cm] at (APP.east) {Stochastic Mirror Descent~\eqref{eq:mirror}};
      \node (sgd) [yshift=-4cm] at (APP.center) {Stochastic Gradient Descent~\eqref{eq:sgd}};
    \end{scope}
    \begin{scope}[line width=1pt, >=stealth, every node/.style={font=\normalsize}]
      \draw[->] (APP.west) to [out=180, in=90] node [pos=0.6, fill=white] {$K = \frac{1}{2}\sqnorm{\cdot}$} (prox.north);
      \draw[->] (APP.east) to [out=0, in=90] node [pos=0.6, fill=white] {$\jad = 0$} (mirror.north);
      \draw[->] (prox.south) to [out=270, in=180] node [pos=0.4, fill=white] {$\jad = 0$} (sgd.west);
      \draw[->] (mirror.south) to [out=270, in=0] node [pos=0.4, fill=white] {$K = \frac{1}{2}\sqnorm{\cdot}$} (sgd.east);
    \end{scope}
  \end{tikzpicture}
  \caption{Links between the different stochastic approximation algorithms.}
  \label{fig:link_algo}
\end{figure}

The paper is organized as follows.  In Section~\ref{sec:desc_ex}, we describe
the setting of the stochastic APP algorithm considered in this paper along with
some examples of application.  In Section~\ref{sec:mes}, we prove the
measurability of the iterates of the stochastic APP algorithm in a
reflexive {separable} Banach space.  The issue of measurability is not often
addressed in the literature, yet it is essential from a theoretical point of
view. When convergence results or efficiency estimates are derived for SA
algorithms, the iterates must be random variables so that the probabilities or
the expectations that appear in the computation are well-defined. For that
purpose, we carry out a precise study based
on~\cite{castaing_convex_1977,hess_measurability_1995} and we adapt some results
of~\cite{rockafellar_variational_2004} to the {infinite-dimensional}
case.  Section~\ref{sec:conv} deals with convergence results and efficiency
estimates. In \S\ref{subsec:conv_app}, convergence results for the iterates and
for the function values of the stochastic APP algorithm are extended to
reflexive {separable} Banach spaces. These results already appear
in~\cite{culioli_decomposition/coordination_1990} for the Hilbert case. They are
also given, again in the Hilbert case, for stochastic projected gradient
in~\cite{geiersbach_projected_2019} and stochastic mirror descent in {the
  finite-dimensional setting}~\cite{nemirovski_robust_2009}. In
\S\ref{subsec:eff_est}, we derive efficiency estimates for the expected function
value taken either for the averaged sequence of iterates or for the last
iterate. These efficiency estimates take into account the additive error on the
subgradient, using the technique from~\cite{geiersbach_stochastic_2020-1}. To
obtain convergence rates for the expected function value of the last iterate, we
adapt the concept of modified Fej\'{e}r monotonicity~\cite{lin_modified_2018} to
the framework of the stochastic APP algorithm.  The paper ends by some
concluding remarks in Section~\ref{sec:ccl}.

\section{Description of the algorithm and examples}
\label{sec:desc_ex}

We describe the version of the stochastic APP algorithm that is studied in this
paper and we give some examples of problems that fit in the general framework of
Problem~\eqref{eq:master_pb}.

\subsection{Setting of the stochastic APP algorithm}
\label{subsec:setting_algo}

The original idea of the APP, first introduced in~\cite{cohen_optimization_1978}
and extended to the stochastic case
in~\cite{culioli_decomposition/coordination_1990}, is to solve a sequence of
auxiliary problems whose solutions converge to the optimal solution of
Problem~\eqref{eq:master_pb}. Assume that $\jc$ is subdifferentiable
{with respect to $u$}. At iteration $k$ of the algorithm, a realization
$w_{k+1}$ of a random variable $\va{W}_{k+1}$ is drawn. The random variables
$\va{W}_{1}, \ldots, \va{W}_{k+1}$ are independent and identically distributed
as $\va{W}$. Then, the following auxiliary problem is solved:
\begin{align}
  \label{pb:stoch-pbsansc-ppa}
  \min_{u \in \Uad} K(u) + \proscal{\varepsilon_{k}(g_{k} + r_{k}) - \nabla K (u_{k})}{u} + \varepsilon_{k} \jad(u, w_{k+1}) \eqfinv
\end{align}
where $g_{k} \in \partial_{u} \jc(u_{k}, w_{k+1})$ and we allow for an additive
error $r_{k}$ on the gradient. The term $r_{k}$ represents a numerical error or
a bias due to an approximation of the gradient e.g. with a finite difference
scheme. The auxiliary problem is characterized by the choice of the auxiliary
function $K$. In the introduction, we have given particular choices for $K$ that
lead to well-known algorithms. Depending on the context, the function $K$ allows
for an adaptation of the algorithm to the geometric structure of the data or it
can provide decomposition properties to the algorithm, see
Example~\ref{ex:decomp_prop}. The stochastic APP algorithm is given in
Algorithm~\ref{alg:sto_ppa}.

\begin{algorithm}[htbp]
  \begin{algorithmic}[1]
    \State Choose an initial point~$u_{0} \in \Uad$, and a positive sequence~$\{ \varepsilon_{k} \}_{k \in \bbN}$.
    \State At iteration~$k$, draw a realization~$w_{k+1}$
    of the random variable~$\va{w}_{k+1}$.
    \State Solve Problem~\eqref{pb:stoch-pbsansc-ppa},
    denote by $u_{k+1}$ the solution.
    \State $k \gets k+1$ and go back to~2.
  \end{algorithmic}
  \caption{Stochastic APP algorithm}
  \label{alg:sto_ppa}
\end{algorithm}

No explicit stopping rule is provided in
Algorithm~\ref{alg:sto_ppa}. It is indeed difficult to know when to stop a
stochastic algorithm as its properties are of statistical nature. Nevertheless,
stopping rules have been developed
in~\cite{wada_stopping_2015,yin_stopping_1990} for the Robbins-Monro
algorithm. In practice, the stopping criterion may be a maximal number of
evaluations imposed by a budget limitation.

\subsection{Some cases of interest for the stochastic APP}
\label{subsec:ex}

The structure of Problem~\eqref{eq:master_pb} is very general and
covers a wide class of problems that arise in machine learning or
stochastic optimal control.  We give some cases of interest that can be cast in
this framework.

\begin{example}
  \emph{Regularized risk minimization in machine learning.}

  Let $(\espacea{X}, \tribu{X})$ and $(\espacea{Y}, \tribu{Y})$ be two
  measurable spaces, where $\tribu{X}$ and $\tribu{Y}$ denote respectively the
  $\sigma$-fields on $\espacea{X}$ and $\espacea{Y}$. Let
  $X \subset \espacea{X}$ and $Y \subset \espacea{Y}$ and assume there is a
  probability distribution $\nu$ on $X \times Y$. Let
  $\na{(x_{i}, y_{i})}_{1 \leq i \leq N} \in (X \times Y)^{N}$ be a training set
  which consists {of} independent and identically distributed samples of
  a random vector $(\va{X}, \va{Y})$ following the distribution $\nu$. Consider
  a convex loss function $\ell : Y \times Y \rightarrow \bbR_{+}$ and let $\ban$
  be a space of functions from $X$ to $Y$. The goal of regularized expected loss
  minimization is to find a regression function $u\opt \in \Uad$, where
  $\Uad \subset \ban$, such that:
  \begin{align}
    \label{eq:expec_loss_min}
    u\opt \in \argmin_{u \in \Uad} \int_{X \times Y} \ell\bp{y, u(x)}\nu(\mathrm{d}x, \mathrm{d}y) + R(u) \eqfinv
  \end{align}
  where $R$ is a regularization term. In practice, as the distribution $\nu$ is
  unknown, we solve an approximate problem, called the regularized empirical
  risk minimization problem:
  \begin{align}
    \label{eq:approx_expec_loss_min}
    u\opt \in \argmin_{u \in \Uad} \frac{1}{N} \sum_{i=1}^{N} \ell\bp{y_{i}, u(x_{i})} + R(u) \eqfinp
  \end{align}
  Problem~\eqref{eq:approx_expec_loss_min} is exactly of the form of
  Problem~\eqref{eq:expec_loss_min} if the distribution $\nu$ is taken to be the
  empirical measure $\nu = 1/N \sum_{i=1}^{N} \delta_{(x_{i}, y_{i})}$, where
  $\delta_{(x_{i}, y_{i})}$ denotes the measure of mass one at $(x_{i}, y_{i})$
  and zero elsewhere. The regularized expected loss minimization
  Problem~\eqref{eq:expec_loss_min} is of the form of
  Problem~\eqref{eq:master_pb} with the smooth term
  $\Jc(u) = \int_{X \times Y} \ell \bp{y, u(x)}\nu(\mathrm{d}x, \mathrm{d}y)$
  and the possibly non-smooth term $\Jad(u) = R(u)$.
\end{example}

\begin{example} \label{ex:decomp_prop}
  \emph{Decomposition aspects of the stochastic APP algorithm.}

  Let $n>0$ be a given positive integer. Suppose that
  $\ban = \ban_{1} \times \ldots \times \ban_{n}$ and
  $\Uad = \Uad_{1} \times \ldots \times \Uad_{n}$ with
  $\Uad_{i} \subset \ban_{i}$ for all $i \in \na{1, \cdots, n}$. Moreover,
  assume that $\jad$ is an additive function, that is,
  $\jad(u, \va{W}) = \sum_{i=1}^{n} \jad_{i}(u^{i}, \va{W})$ with
  $u^{i} \in \ban_{i}$, whereas $\jc$ induces a non-additive coupling. In this
  case, Problem~\eqref{eq:master_pb} is:
  \begin{align*}
    \min_{u \in \Uad} \Jc(u) + \sum_{i=1}^{n} \Jad_{i}(u^{i}) \eqfinv
  \end{align*}
  where $\Jad_{i}(u^{i}) = \expec{\jad_{i}(u^{i}, \va{W})}$. We apply the
  stochastic APP algorithm with an additive auxiliary function
  $K(u) = \sum_{i=1}^{n} K_{i}(u^{i})$. Let $\bar{u} \in \ban$ be given, a
  canonical choice for $K_{i}$ is:
  \begin{align*}
    K_{i}(u^{i}) = \Jc(\bar{u}^{1:i-1}, u^{i}, \bar{u}^{i+1:n}), \quad i \in \na{1, \ldots, n}  \eqfinv
  \end{align*}
  where $\bar{u}^{i:j} = (\bar{u}^{i}, \ldots, \bar{u}^{j})$ for
  $1\leq i \leq j \leq n$ and $\bar{u}^{1:0}$ denotes the empty vector by
  convention. Another classical choice is $K = \frac{1}{2}\sqnorm{\cdot}$. With
  an additive function $K$, the auxiliary problem~\eqref{pb:stoch-pbsansc-ppa}
  can be split into $n$ independent subproblems that can be solved in
  parallel. At iteration $k$ of the stochastic APP algorithm, the $i$-th
  subproblem is:
  \begin{align}
    \min_{u^{i} \in \Uad_{i}} K_{i}(u^{i}) + \proscal{\varepsilon_{k}(g_{k}^{i} + r_{k}^{i}) - \nabla K_{i} (u_{k}^{i})}{u^{i}} + \varepsilon_{k} \jad_{i}(u^{i}, w_{k+1}) \eqfinv
    \label{eq:stoch-ppa_subpb}
  \end{align}
  where $g_{k}^{i} \in \partial_{u^{i}} \jc(u_{k}, w_{k+1})$ and $r_{k}^{i}$ is an additive error on $\partial_{u^{i}} \jc(u_{k}, w_{k+1})$. This example shows that the stochastic APP encompasses decomposition techniques.
\end{example}

\section{Measurability of the iterates of the stochastic APP algorithm}
\label{sec:mes}

Convergence results for SA algorithms often consist in proving the almost sure
convergence of the sequence of iterates $\ba{\va{U}_{k}}_{k \in \bbN}$ to the
optimal value $u\opt$. Other results provide non-asymptotic bounds for the
expectation of function values $\espe\bp{\J(\va{U}_{k}) -\J(u\opt)}$ or the
quadratic mean $\expec{\sqnorm{\va{U}_{k} - u\opt}}$ for example.  In order for
these expectations and probabilities to be well-defined, $\va{U}_{k}$ must be a
measurable mapping from $\omeg$ to $\ban$. However, {as far as we know,
  the current literature does not provide constructive conditions under which
  the measurability of $\va{U}_{k}$ is ensured}. {We aim at filling this
  theoretical gap by proving} the measurability of the iterates of the
stochastic APP algorithm.

\subsection{A general measurability result}

This section is devoted to the proof of a general measurability result in
Theorem~\ref{thm:mes}. We obtain the measurability of the iterates of the
stochastic APP algorithm as a consequence in
Theorem~\ref{thm:mes_selec_sto_app}.

Recall that $(\omeg, \trib, \prbt)$ is a probability space and that
$(\messpace, \borel{\messpace})$ is a measurable topological vector space. The
Banach space $\ban$ is equipped with the Borel $\sigma$-field
$\borel{\ban}$. The topological dual of $\ban$, {equipped with the
  topology of the norm induced by the primal norm}, is denoted by $\ban\dual$,
and its Borel $\sigma$-field is $\borel{\ban\dual}$. We consider the following
problem:
\begin{align}
  \min_{u \in \Uad} \left\{\Phi(\omega, u) \defegal K(u) + \proscal{\va{\varphi}(\omega)}{u} + \varepsilon \jad\bp{u, \va{W}(\omega)} \right\}\eqfinv
  \label{eq:pb_interest}
\end{align}
where $\varepsilon > 0$ is a given positive real {number} and
$\va{\varphi} : \omeg \rightarrow \ban\dual$ is a given measurable function. The
goal is to show the existence of a measurable mapping $\widetilde{\va{U}}$ such
that for all
$\omega \in \omeg, \ \widetilde{\va{U}}(\omega) \in \argmin_{u \in \Uad}
\Phi(\omega, u)$. The mapping
$\omega \mapsto \argmin_{u \in \Uad} \Phi(\omega, u)$ is a set-valued mapping.

\subsubsection{Some tools from the theory of set-valued mappings}
\label{subsubsec:tools}

We recall some results from the theory of set-valued mappings that are used to
state and prove the measurability result of Theorem~\ref{thm:mes}. The
definitions and propositions are mostly taken
from~\cite{castaing_convex_1977,hess_measurability_1995}.
{Theorem~\ref{thm:mes} requires
$\ban$ to be a reflexive separable Banach space. However, all results
from~\S\ref{subsubsec:tools} are more generally valid for $\ban$ being
a Polish space.}
For two sets $X, \ Y$, we denote by
$\Gamma : X \rightrightarrows Y$ a set-valued mapping $\Gamma$ from $X$ to
$Y$. This means that for $x \in X,\ \Gamma(x) \subset Y$ or in other words that
$\Gamma(x) \in \mathcal{P}(Y)$ where $\mathcal{P}(Y)$ is the power set of $Y$.

\begin{definition}[Measure completion]
  Let $(\omeg, \trib)$ be a measurable space.
  \begin{itemize}
  \item Let $\mu$ be a measure on $(\omeg, \trib)$. The \emph{$\mu$-completion} of
    $\trib$ is the $\sigma$-field $\trib_{\mu}$ generated by
    $\trib \cup \na{A' \in \mathcal{P}(\omeg) \,\vert\, A' \subset A, A \in \trib \text{ and } \mu(A) = 0}$, that is, the union of $\trib$
    and the $\mu$-negligible sets. The $\sigma$-field $\trib$ is said to be \emph{complete} for
    the measure $\mu$ if  $\trib = \trib_{\mu}$.
    \item The $\sigma$-field $\hat{\trib}$ of universally measurable sets is defined by $\hat{\trib} = \bigcap_{\mu} \trib_{\mu}$ where $\mu$ ranges over the set of positive $\sigma$-finite measures on the measurable space $(\omeg, \trib)$.
  \end{itemize}
\end{definition}

\begin{definition}[Measurable selection]
  Let $(\omeg, \trib)$ be a measurable space and $\ban$ be a separable
  {Banach} space. Let $\Gamma : \omeg \rightrightarrows \ban$ be a
  set-valued mapping. A function $\gamma : \omeg \rightarrow \ban$ is a
  measurable selection of $\Gamma$ if $\gamma(\omega) \in \Gamma(\omega)$ for
  all $\omega \in \omeg$ and $\gamma$ is measurable.
\end{definition}

\begin{definition} [Measurable mapping]
  \label{def:mes}
  Let $(\omeg, \trib)$ be a measurable space and $\ban$ be a
  separable {Banach} space. A set-valued mapping $\Gamma : \Omega \rightrightarrows \ban$ is
  Effros-measurable if, for every open set $O \subset \ban$, we have:
  \begin{align*}
    \Gamma^{-}(O) = \ba{\omega \in \omeg, \ \Gamma(\omega) \cap O \neq \emptyset} \in \trib \eqfinp
  \end{align*}
\end{definition}

\begin{remark}
  The Effros-measurability of a set-valued mapping $\Gamma : \omeg \rightrightarrows \ban$
  is equivalent to the measurability of $\Gamma$ viewed as a function from $\omeg$ to $\mathcal{P}(\ban)$.
\end{remark}

\begin{proposition}\cite[Theorem III.9]{castaing_convex_1977}
  \label{prop:cast_rep}
  Let $(\omeg, \trib)$ be a measurable space and $\ban$ be a separable Banach
  space. Let $\Gamma : \omeg \rightrightarrows \ban$ be a non-empty-valued and
  closed-valued mapping. Then the following statements are equivalent:
  \begin{enumerate}[label=(\roman*)]
  \item $\Gamma$ is
    Effros-measurable.
  \item $\Gamma$ admits a Castaing representation: there exists a sequence of
    measurable functions $\na{\gamma_{n}}_{n \in \bbN}$ such that for all
    $\omega \in \omeg, \ \Gamma(\omega) = \cl\ba{\gamma_{n}(\omega), \ n \in
      \bbN}$ where $\cl$ denotes the closure of a set.
  \end{enumerate}
\end{proposition}

\begin{remark}
  { An important consequence of Proposition~\ref{prop:cast_rep} is that
    any Effros-measurable mapping admits a measurable selection. This result is
    usually known as the Kuratowski–Ryll-Nardzewski selection
    theorem~\cite{kuratowski_general_1965}.  }
\end{remark}

\begin{proposition}\cite[Proposition III.23: Sainte-Beuve’s projection
  theorem]{castaing_convex_1977} Let $(\omeg, \trib)$ be a measurable space and
  $(\ban, \borel{\ban})$ be a separable Banach space equipped with its Borel
  $\sigma$-field. Let $G \in \trib \otimes \borel{\ban}$. Denote by
  $\proj{\omeg}{G}$ the projection of $G$ on $\omeg$. Then,
  $\proj{\omeg}{G} \in \hat{\trib}$, where we recall that $\hat{\trib}$ is the
  $\sigma$-field of universally measurable sets.
  \label{prop:proj_thm}
\end{proposition}

\begin{proposition}\cite[Proposition III.30]{castaing_convex_1977} \label{prop:mes_closed}
  Let $(\omeg, \trib, \mu)$ be a measure space where $\trib$ is a complete
  $\sigma$-field, that is, $\trib = \trib_{\mu}$ and let $\ban$ be a separable
  Banach space. Let $\Gamma : \omeg \rightrightarrows \ban$ be a non-empty
  valued and closed-valued mapping. The following statements are equivalent:
  \begin{enumerate}[label=(\roman*)]
    \item $\Gamma$ is Effros-measurable.\label{effros_mes}
    \item \label{closed_mes} For every closed set $C \subset \ban$, we have:
    \begin{align*}
      \Gamma^{-}(C) = \na{\omega \in \omeg, \ \Gamma(\omega) \cap C \neq \emptyset} \in \trib \eqfinp
    \end{align*}
  \end{enumerate}
\end{proposition}

\begin{remark}
  When $\espacea{U}$ is finite-dimensional, Proposition~\ref{prop:mes_closed} is
  true in any measurable space $(\omeg, \trib)$; that is, the completeness
  assumption of the $\sigma$-field $\trib$ is not needed \cite[Theorem
  14.3]{rockafellar_variational_2004}. In the infinite-dimensional setting,
  \ref{closed_mes} implies~\ref{effros_mes} remains true in any measurable space
  $(\omeg, \trib)$~\cite[Proposition III.11]{castaing_convex_1977}. The
  completeness assumption is only required to prove \ref{effros_mes}
  implies~\ref{closed_mes} when $\espacea{U}$ is
  infinite-dimensional. Essentially, in the finite-dimensional case, the proof
  of \ref{effros_mes} implies~\ref{closed_mes} relies on the fact that
  $\espacea{U}$ is locally compact. In the infinite-dimensional case,
  $\espacea{U}$ is not locally compact and the proof uses the Sainte-Beuve's
  projection theorem.
\end{remark}

\begin{definition} [Graph and epigraph]
  Let $(\espacea{X}, \tribu{X})$ be a measurable space and $\ban$ be a Banach
  space. Let $h: \espacea{X} \rightarrow \bbR \cup \{+\infty\}$ be a function
  and $\Gamma: \espacea{X} \rightrightarrows \ban$ be a set-valued mapping.
  \begin{itemize}
    \item The graph and the epigraph of $h$ are respectively defined by:
    \begin{align*}
      \gph h &= \ba{(x, \alpha) \in \espacea{X} \times \bbR, \ h(x) = \alpha}\eqfinv\\
      \epi h &= \ba{(x, \alpha) \in \espacea{X} \times \bbR, \ h(x) \leq \alpha}\eqfinp
    \end{align*}
    \item The graph of $\Gamma$ is defined by:
    \begin{align*}
      \gph \Gamma = \ba{(x, u) \in \espacea{X} \times \ban, \ u \in \Gamma(x)}\eqfinp
    \end{align*}
  \end{itemize}
\end{definition}

\begin{definition} [Normal integrand]
  \label{def:norm_int}
  Let $(\omeg, \trib)$ be a measurable space and $\ban$ be a Banach space. A
  function $f: \omeg \times \ban \rightarrow \bbR \cup \{+\infty\}$ is a normal
  integrand if it satisfies the following conditions:
  \begin{enumerate}[label=(\roman*)]
  \item\label{lsc_norm_int} For all $\omega \in \omeg, \ f(\omega, \cdot)$ is \lsc,
  \item\label{epi_mes} The epigraphical mapping
    $S_{f} : \omeg \rightrightarrows \ban \times \bbR$ defined by
    $S_{f}(\omega) = \epi f(\omega, \cdot)$ is
    Effros-measurable.
  \end{enumerate}
\end{definition}

\begin{remark}
  The point~\ref{lsc_norm_int} of Definition~\ref{def:norm_int} is equivalent to
  $S_{f}$ being closed-valued. In this paper, we consider the definition of the
  normal integrand used by Hess~\cite{hess_measurability_1995}. It differs from
  the definition of Castaing~\cite{castaing_convex_1977} where the
  point~\ref{epi_mes} is replaced by the
  $\trib \otimes \borel{\ban}$-measurability of $f$. We shall see in
  Proposition~\ref{prop:joint_mes} that the Effros-measurability of the
  epigraphical mapping $S_{f}$ implies the
  $\trib \otimes \borel{\ban}$-measurability of $f$. Note also that if $\trib$
  is complete for a positive $\sigma$-finite measure $\mu$, these two
  definitions are equivalent, see~\cite[Proposition
  III.30]{castaing_convex_1977}.
\end{remark}

\begin{definition} [Carath\'{e}odory integrand]
  Let $(\omeg, \trib)$ be a measurable space and $\ban$ be a separable Banach
  space. A function $f : \omeg \times \ban \rightarrow \bbR$ (finite-valued) is
  a Carath\'{e}odory integrand if it satisfies the following conditions:
  \begin{enumerate}[label=(\roman*)]
    \item For all $u \in \ban, \ f(\cdot, u)$ is measurable.
    \item For all $\omega \in \omeg, \ f(\omega, \cdot)$ is continuous.
  \end{enumerate}
\end{definition}

\begin{proposition}\cite[Proposition 2.5]{hess_measurability_1995}
  If $f$ is a Carath\'{e}odory integrand, then it is a normal integrand.
\end{proposition}



\begin{proposition}\cite[Proposition III.13]{castaing_convex_1977}
  Let $(\omeg, \trib)$ be a measurable space and $(\ban, \borel{\ban})$ be a
  separable Banach space equipped with its Borel $\sigma$-field. If
  $\Gamma : \omeg \rightrightarrows \ban$ is an Effros-measurable, closed-valued
  mapping, then $\gph \Gamma \in \trib \otimes \borel{\ban}$.
  \label{prop:gph_joint_mes}
\end{proposition}

We now recall a technical result on the Borel $\sigma$-field of a product space
that is used in the proof of subsequent propositions.

\begin{proposition}\cite[Proposition 7.13]{bertsekas_stochastic_1996}
  \label{prop:borel_prod}
  Let $\ba{\bp{\espacea{X}_{i}, \borel{\espacea{X}_{i}}}}_{i \in \bbN}$ be a
  sequence of measurable separable topological spaces equipped with their Borel
  $\sigma$-fields. For $n \in \bbN$, let
  $\espacea{Y}_{n} = \prod_{i=1}^{n} \espacea{X}_{i}$ and let
  $\espacea{Y} = \prod_{i \in \bbN} \espacea{X}_{i}$. Then, the Borel
  $\sigma$-field of the product space $\espacea{Y}_{n}$ (resp. $\espacea{Y}$)
  coincides with the product of the Borel $\sigma$-fields of
  $\na{\espacea{X}_{i}}_{i=1}^{n}$ (resp. $\na{\espacea{X}_{i}}_{i\in \bbN}$),
  that is:
  \begin{align*}
    \tribu{B}\vardelim{\espacea{Y}_{n}} = \bigotimes_{i=1}^{n} \borel{\espacea{X}_{i}} \quad \text{and} \quad \tribu{B}\vardelim{\espacea{Y}} = \bigotimes_{i \in \bbN} \borel{\espacea{X}_{i}}\eqfinp
  \end{align*}
\end{proposition}

The following proposition shows that a normal integrand
$f: \omeg \times \ban \rightarrow \bbR \cup \{+\infty\}$, as defined in
\cite{hess_measurability_1995}, is jointly
$\trib \otimes \borel{\ban}$-measurable. This result is given in~\cite[Corollary
14.34]{rockafellar_variational_2004} when $\espacea{U} = \bbR^{n}$ but is
extended here in the Banach case.

\begin{proposition}
  Let $(\omeg, \trib)$ be a measurable space and $(\ban, \borel{\ban})$ be a
  separable Banach space equipped with its Borel $\sigma$-field. If
  $f: \omeg \times \ban \rightarrow \bbR \cup \{+\infty\}$ is a normal
  integrand, then $f$ is $\trib \otimes \borel{\ban}$-measurable.
  \label{prop:joint_mes}
\end{proposition}

\begin{proof}
  The function $f$ is a normal integrand so its epigraphical mapping $S_{f}$ is
  Effros-measurable and closed-valued. Moreover $\ban$ is separable, so by
  Proposition~\ref{prop:gph_joint_mes}, we get that:
  \begin{align*}
    \gph S_{f} = \ba{(\omega, u, \alpha) \in \omeg \times \ban \times \bbR, \ f(\omega, u) \leq \alpha} \in \trib \otimes \borel{\ban \times \bbR}\eqfinp
  \end{align*}
  Using that $\ban$ and $\bbR$ are separable, we have
  $\borel{\ban \times \bbR} = \borel{\ban} \otimes \borel{\bbR}$ by
  Proposition~\ref{prop:borel_prod}. Then, for each $\alpha \in \bbR$, we get:
  \begin{align*}
    f^{-1}\bp{]-\infty, \alpha]} = \ba{(\omega, u) \in \omeg \times \ban, \ f(\omega, u) \leq \alpha} \in \trib \otimes \borel{\ban}\eqfinp
  \end{align*}
  This shows that $f$ is $\trib \otimes \borel{\ban}$-measurable.
\end{proof}

The following proposition is an adaptation of~\cite[Proposition
14.45(c)]{rockafellar_variational_2004} on the composition operations on normal
integrands to the Banach case. The separability of $\ban$ is a crucial
assumption that is used explicitly in the proof of
Proposition~\ref{prop:mes_comp} and that appears in most of the results of this
part. Essentially, as only a countable union of measurable sets is measurable,
countable dense subsets of a separable space are often used in proofs of
measurability. Moreover, in the infinite-dimensional setting, we must assume the
completeness of the $\sigma$-field $\trib$ because we appeal to
Proposition~\ref{prop:mes_closed} in the proof.

\begin{proposition}
  Let $(\omeg, \trib, \mu)$ be a measure space where $\trib$ is a complete
  $\sigma$-field, that is, $\trib = \trib_{\mu}$. Let
  $(\messpace, \borel{\messpace})$ be a topological measurable space and
  $(\ban, \borel{\ban})$ be a separable Banach space equipped with its Borel
  $\sigma$-field. Let
  $h : \ban \times \messpace \rightarrow \bbR \cup \{+\infty\}$ be \lsc\ and
  $\va{W}: \omeg \rightarrow \messpace$ be a measurable mapping. Then:
  \begin{align*}
    f: (\omega, u) \in \omeg \times \ban \mapsto h\bp{u, \va{W}(\omega)} \in \bbR \cup \{+\infty\}
  \end{align*}
  is a normal integrand.
  \label{prop:mes_comp}
\end{proposition}

\begin{proof}
  We have that $h$ is \lsc\ so $f(\omega, \cdot) = h(\cdot, \va{W}(\omega))$ is
  \lsc\ for all $\omega \in \omeg$. It remains to prove that the epigraphical
  mapping $S_{f}$ is Effros-measurable.  As $h$ is \lsc, the set $\epi h$ is
  closed.  Define:
  \begin{align*}
    G: (\omega, u, \alpha) \in \omeg \times \ban \times \bbR \mapsto (u, \va{W}(\omega), \alpha) \in \ban \times \messpace \times \bbR \eqfinp
  \end{align*}
  Then, let:
  \begin{align*}
    Q(\omega) &= \bc{(\ban \times \bbR) \times \epi h} \cap \gph G(\omega, \cdot, \cdot) \eqfinv \\
    &= \ba{\vardelim{(u, \alpha), (u, \va{W}(\omega), \alpha)} \ \text{such that} \ h(u, \va{W}(\omega)) \leq \alpha, \ (u, \alpha) \in \espacea{U} \times \bbR} \eqfinv\\
    &= \ba{\vardelim{(u, \alpha), (u, \va{W}(\omega), \alpha)} \ \text{such that} \ f(\omega, u) \leq \alpha, \ (u, \alpha) \in \espacea{U} \times \bbR} \eqfinp
  \end{align*}
  Now, define the projection operator $P$ as:
  \begin{align*}
    P : (\ban \times \bbR) \times (\ban \times \messpace \times \bbR) &\rightarrow (\ban \times \bbR) \eqfinv\\
    ((u, \alpha), (v, w, \beta)) & \mapsto (u, \alpha)
  \end{align*}
  so that we have:
  \begin{align*}
    S_{f}(\omega) = \ba{(u, \alpha) \in \espacea{U} \times \bbR, \ f(\omega, u) \leq \alpha} = P\bp{Q(\omega)}\eqfinp
  \end{align*}
  \begin{itemize}
  \item Let $\Gamma$ be the set valued mapping defined by
    $\Gamma: \omega \in \omeg \mapsto \gph G(\omega, \cdot, \cdot) \in (\ban
    \times \bbR) \times (\ban \times \messpace \times \bbR)$. We show that
    $\Gamma$ is Effros-measurable. As $\ban$ is separable, there exists a
    countable dense subset $\na{(b_{n}, r_{n}), \ n \in \bbN}$ of
    $\ban \times \bbR$. For $n \in \bbN$, let
    $\gamma_{n}(\omega) = \bp{(b_{n}, r_{n}), G(\omega, b_{n}, r_{n})}$. As
    $G(\omega, b_{n}, r_{n}) = (b_{n}, \va{W}(\omega), r_{n})$ and $\va{W}$ is
    measurable, we get that $\gamma_{n}$ is measurable. Then, we have
    $\Gamma(\omega) = \cl\na{\gamma_{n}(\omega), \ n \in \bbN}$. Hence,
    $\na{\gamma_{n}}_{n \in \bbN}$ is a Castaing representation of
    $\Gamma$. Moreover, $\Gamma$ is closed-valued and non-empty valued so by
    Proposition~\ref{prop:cast_rep}, we deduce that $\Gamma$ is
    Effros-measurable.
  \item Let $C \subset (\ban \times \bbR) \times (\ban \times \messpace \times \bbR)$ be a closed set. We have:
  \begin{align*}
    Q^{-}(C) &= \ba{\omega \in \omeg, \bc{(\ban \times \bbR) \times \epi h} \cap \Gamma(\omega) \cap C \neq \emptyset} \eqfinv\\
    &= \Gamma^{-}\vardelim{C \cap \bc{(\ban \times \bbR) \times \epi h}} \eqfinp
  \end{align*}
  As $\epi h$ is closed, the set $C \cap \bc{(\ban \times \bbR) \times \epi h}$
  is closed. By assumption, the $\sigma$-field $\trib$ is complete and we have
  shown that $\Gamma$ is Effros-measurable, therefore by
  Proposition~\ref{prop:mes_closed}, we get that
  $\Gamma^{-}\vardelim{C \cap \bc{(\ban \times \bbR) \times \epi h}} = Q^{-}(C)
  \in \trib$. Hence, $Q$ is Effros-measurable.
  \item Finally, for every open set $V \subset \ban \times \bbR$, as $S_{f}(\omega) = P\bp{Q(\omega)}$, we have:
    \begin{align*}
      S_{f}^{-}(V) = \ba{\omega \in \omeg, \ Q(\omega) \cap P^{-1}(V) \neq \emptyset}\eqfinp
    \end{align*}
    The projection $P$ is continuous so $P^{-1}(V)$ is open. As $Q$ is
    Effros-measurable, we get that $S_{f}^{-}(V) \in \trib$, that is, $S_{f}$ is
    Effros-measurable.
  \end{itemize}
  This completes the proof.
\end{proof}

We now give the main results that are used to prove the measurability of the
iterates of the stochastic APP. The following proposition is a slight extension
of~\cite[Proposition 4.2(c)]{hess_epi-convergence_1996}.

\begin{proposition}
  Let $(\omeg, \trib)$ be a measurable space and $(\ban, \borel{\ban})$ be a
  separable Banach space equipped with its Borel $\sigma$-field. Let $\Uad$ be a
  closed subset of $\ban$. Let
  $f: \omeg \times \ban \rightarrow \bbR \cup \{+\infty\}$ be an
  $\trib \otimes \borel{\ban}$-measurable function. Let
  $M: \omeg \rightrightarrows \ban$ be the argmin set-valued mapping:
  \begin{align*}
    M(\omega) = \argmin_{u \in \Uad} f(\omega, u)\eqfinp
  \end{align*}
  Assume that the argmin mapping $M$ is non-empty valued, then $M$ admits an $\hat{\trib}$-measurable selection.
  \label{prop:mes_argmin}
\end{proposition}

\begin{proof}
  Let $\alpha \in \bbR$ and $m(\omega) = \min_{u \in \Uad} f(\omega, u)$. The function $m$ is well-defined as $M$ is non-empty valued. Let:
  \begin{align*}
    H = (\omeg \times \Uad) \cap \left\{(\omega, u) \in \omeg \times \ban, \ f(\omega, u) < \alpha\right\} \eqfinp
  \end{align*}
  We have:
  \begin{align*}
    \left\{\omega \in \omeg, \ m(\omega) < \alpha\right\} = \proj{\omeg}{H} \eqfinv
  \end{align*}
  where $\proj{\omeg}{H}$ is the projection of $H$ on $\omeg$. As $f$ is
  $\trib \otimes \borel{\ban}$-measurable and $\Uad$ is closed hence measurable,
  we get that $H \in \trib \otimes \borel{\ban}$. From
  Proposition~\ref{prop:proj_thm}, we deduce that
  $m^{-1}\vardelim{]-\infty, \alpha[}$ is $\hat{\trib}$-measurable so that $m$
  is $\hat{\trib}$-measurable. As $\trib \subset \hat{\trib}$, {we have
    that
    $\mathcal{A} \otimes \mathcal{B}(\mathbb{U}) \subset \hat{\mathcal{A}}
    \otimes \mathcal{B}(\mathbb{U})$, therefore} the function $f$ is
  $\hat{\trib} \otimes \borel{\ban}$-measurable. We can write:
  \begin{align*}
    M(\omega) = \left\{u \in \Uad, \ f(\omega, u) = m(\omega)\right\} \eqfinv
  \end{align*}
  so,
  $\gph M = \left\{(\omega, u) \in \omeg \times \Uad, f(\omega, u) =
    m(\omega)\right\}$. Therefore, $\gph M$ is
  $\hat{\trib} \otimes \borel{\ban}$-measurable as the inverse image of $\na{0}$
  under the $\hat{\trib} \otimes \borel{\ban}$-measurable mapping
  $(\omega, u) \mapsto f(\omega, u) - m(\omega)$. Let $O$ be an open subset of
  $\ban$. We have:
  \begin{align*}
    M^{-}(O) = \proj{\omeg}{(\omeg \times O) \cap \gph M} \eqfinp
  \end{align*}
  As $(\omeg \times O) \cap \gph M \in \hat{\trib} \otimes \borel{\ban}$, by
  Proposition~\ref{prop:proj_thm}, we get that
  $M^{-}(O) \in \hat{\hat{\trib}} = \hat{\trib}$. Hence, $M$ is
  Effros-measurable for the $\sigma$-field $\hat{\trib}$ and is non-empty-valued
  by assumption, so by Proposition~\ref{prop:cast_rep}, $M$ admits an
  $\hat{\trib}$-measurable selection.
\end{proof}

\begin{corollary}
  Let $(\omeg, \trib, \mu)$ be a complete probability space, i.e. $\trib = \trib_{\mu}$. Let $(\ban, \borel{\ban})$ be a separable Banach space equipped with its Borel $\sigma$-field. Let $f: \omeg \times \ban \rightarrow \bbR \cup \{+\infty\}$ be an $\trib \otimes \borel{\ban}$-measurable function. Suppose that the argmin mapping $M$ is non-empty valued. Then, $M$ admits an $\trib$-measurable selection.
  \label{cor:mes_argmin}
\end{corollary}

\begin{proof}
  As $\mu$ is a positive $\sigma$-finite measure, we have
  $\hat{\trib} = \bigcap_{\mu} \trib_{\mu} \subset \trib_{\mu} = \trib$. By
  Proposition~\ref{prop:mes_argmin}, $M$ admits an $\hat{\trib}$-measurable
  selection, which is also an $\trib$-measurable selection.
\end{proof}

\begin{proposition}\cite[Theorem 4.6]{hess_measurability_1995}
  Let $(\omeg, \trib)$ be a measurable space and $\ban$ be a separable Banach
  space with separable topological dual $\ban\dual$. Let
  $f: \omeg \times \ban \rightarrow \bbR \cup \{+\infty\}$ be a {proper}
  normal integrand {and} $\va{U} : \omeg \rightarrow \ban$ be a
  measurable mapping. Then, the set-valued mapping
  $D_{\va{U}}: \omeg \rightrightarrows \ban\dual$ {defined by}
  \begin{align*}
    D_{\va{U}}(\omega) &= \partial_{u} f\bp{\omega, \va{U}(\omega)} \\
    &= \ba{v \in \ban\dual, \ f(\omega, u) \geq f(\omega, \va{U}(\omega)) + \proscal{v}{u - \va{U}(\omega)}, \ \forall u \in \ban}\eqfinv
  \end{align*}
  is Effros-measurable.
  \label{prop:mes_diff}
\end{proposition}

\subsubsection{Existence of a measurable selection for the \texorpdfstring{$\mathrm{argmin}$}{argmin} mapping of \texorpdfstring{$\Phi$}{Phi}}
\label{subsubsec:mes_selec_argmin}

In this section, we make use of the tools introduced in~\S \ref{subsubsec:tools}
to prove our main measurability result. We introduce the argmin set-valued
mapping $M : \omeg \rightrightarrows \ban$ for Problem~\eqref{eq:pb_interest}:
\begin{align}
  M(\omega) = \argmin_{u \in \Uad} \left\{\Phi(\omega, u) \defegal K(u) + \proscal{\va{\varphi}(\omega)}{u} + \varepsilon \jad\bp{u, \va{W}(\omega)} \right\}\eqfinp
  \label{eq:aux_pb_indep}
\end{align}
We consider the following assumptions:
\begin{enumerate}[label=(A\arabic*)]
  \item The space $\ban$ is a reflexive, separable Banach space. {This implies in particular that $\ban\dual$ is separable.}\label{hyp:sep_banach}
  \item $\Uad$ is a non-empty closed convex subset of $\ban$.\label{hyp:uad}
  \item $\jad: \ban \times \messpace \rightarrow \bbR$ is jointly \lsc\ and for all $w \in \messpace$, $\jad(\cdot, w)$ is proper and convex. \label{hyp:jad}
  \item The function $K: \ban \rightarrow \bbR$ is proper, convex, \lsc\ and Gateaux-differentiable on an open set containing $\Uad$. \label{hyp:k}
  \item For all $\omega \in \omeg$, the function $u \mapsto \Phi(\omega, u)$ is coercive on $\Uad$ meaning that when $\norm{u} \rightarrow +\infty$ with $u \in \Uad$, we have $\Phi(\omega, u) \rightarrow +\infty$. This assumption is automatically satisfied if $\Uad$ is bounded. \label{hyp:phi_coer}
  \item The $\sigma$-field $\trib$ is complete for the measure $\prbt$, that is, $\trib = \trib_{\prbt}$. \label{hyp:complete}
  \item The function $\va{W} : \omeg \rightarrow \messpace$ is measurable. \label{hyp:mes_w}
  \item The function $\va{\varphi}: \omeg \rightarrow \ban\dual$ is measurable. \label{hyp:mes_phi}
\end{enumerate}
The objective of this part is to prove that $M$ defined in
Equation~\eqref{eq:aux_pb_indep} admits a measurable selection.  We start by a
classical theorem from optimization theory giving conditions for the existence
and {uniqueness} of a minimizer $\Phi(\omega, \cdot)$.

\begin{theorem}
  Let $\omega \in \omeg$. Under
  Assumptions~\textup{\ref{hyp:sep_banach}-\ref{hyp:phi_coer}}, $M(\omega)$ is
  non-empty, closed and convex. Moreover, if $K$ is strongly convex, then
  $M(\omega)$ is a singleton, meaning that $\Phi(\omega, \cdot)$, defined
  in~\eqref{eq:aux_pb_indep}, has a unique minimizer.
  \label{thm:argmin_non_empty}
\end{theorem}

\begin{proof}
  The objective function $\Phi(\omega, \cdot)$ is the sum of three convex, \lsc
  \ functions, it is then convex and \lsc\ By~\ref{hyp:phi_coer},
  $\Phi(\omega, \cdot)$ is also coercive. As $\espacea{U}$ is a reflexive Banach
  space~\ref{hyp:sep_banach} and $\Uad$ is non-empty, closed and
  convex~\ref{hyp:uad}, the set of minimizers $M(\omega)$ is non-empty
  \cite[Corollary III.20]{brezis_analyse_2005}. The convexity of
  $\Phi(\omega, \cdot)$ ensures that $M(\omega)$ is convex and the
  lower-semicontinuity of $\Phi(\omega, \cdot)$ ensures that $M(\omega)$ is
  closed.

  If $K$ is strongly convex, then $\Phi(\omega, \cdot)$ is strongly convex,
  hence the minimizer of $\Phi(\omega, \cdot)$ is unique so $M(\omega)$ is a
  singleton.\footnote{In the case where $K$ is strongly convex, the coercivity
    assumption is not needed as it is implied by the strong convexity of
    $\Phi(\omega, \cdot)$.}
\end{proof}

\begin{theorem}
  \label{thm:mes}
  Under Assumptions~\textup{\ref{hyp:sep_banach}-\ref{hyp:mes_phi}}, the mapping
  $M$ defined in Equation~\eqref{eq:aux_pb_indep} admits a measurable selection.
\end{theorem}

\begin{proof}
  We start by proving that
  $\Phi(\omega, u) = K(u) + \proscal{\va{\varphi}(\omega)}{u} + \varepsilon
  \jad\bp{u, \va{W}(\omega)}$ is a normal integrand:
  \begin{itemize}
  \item As the function $K$ is \lsc~\ref{hyp:k}, $(\omega, u) \mapsto K(u)$ is a
    normal integrand. Indeed, its epigraphical mapping
    $\omega \mapsto \na{(u, \alpha) \in \ban \times \bbR,\ K(u) \leq \alpha}$ is
    a constant function of $\omega$ and is then measurable.
  \item The Banach space $\ban$ is separable~\ref{hyp:sep_banach} and $\trib$ is
    complete~\ref{hyp:complete}. The space $\ban\dual$ equipped with its Borel
    $\sigma$-field $\borel{\ban\dual}$ is a measurable space. The function
    $\va{\varphi}$ is measurable~\ref{hyp:mes_phi} and the function
    $(u, v) \in \ban \times \ban\dual \mapsto \proscal{v}{u} \in \bbR$ is
    continuous
    hence, in particular, \lsc. Then, Proposition~\ref{prop:mes_comp} applies,
    showing that the function
    $(\omega, u) \mapsto \proscal{\va{\varphi}(\omega)}{u}$ is a normal
    integrand.
  \item With the same reasoning, using that $\ban$ is separable
    \ref{hyp:sep_banach}, $\va{W}$ is measurable~\ref{hyp:mes_w}, $\trib$ is
    complete~\ref{hyp:complete} and $\jad$ is \lsc\ \ref{hyp:jad}, we use
    Proposition~\ref{prop:mes_comp} with $h = \jad$ to deduce that
    $(\omega, u) \mapsto \jad\bp{u, \va{W}(\omega)}$ is a normal integrand.
  \end{itemize}
  The function $\Phi$ is then a normal integrand as the sum of three normal
  integrands. {As $\ban$ is separable~\ref{hyp:sep_banach}, we use}
  Proposition~\ref{prop:joint_mes}  {to get that} $\Phi$ is
  $\trib \otimes \borel{\ban}$-measurable. In addition,
  using~\ref{hyp:uad}-\ref{hyp:phi_coer} to apply
  Theorem~\ref{thm:argmin_non_empty} ensures that $M$ is non-empty
  valued. Moreover, the $\sigma$-field $\trib$ is complete for
  $\prbt$~\ref{hyp:complete}. Hence, by Corollary~\ref{cor:mes_argmin}, we
  conclude that $M: \omega \mapsto \argmin_{u \in \Uad} \Phi(\omega, u)$ admits
  a measurable selection.
\end{proof}

\begin{corollary}
  Under Assumptions~\textup{\ref{hyp:sep_banach}-\ref{hyp:mes_phi}} and if we
  additionally assume that $K$ is strongly convex, then for all
  $\omega \in \omeg, \ \Phi(\omega, \cdot)$, defined in~\eqref{eq:aux_pb_indep},
  has a unique minimizer and the mapping:
  \begin{align*}
    \optaux(\omega) =  \argmin_{u \in \Uad} \Phi(\omega, u) \in \ban
  \end{align*}
  is measurable, that is, $\optaux$ is a random variable.
  \label{cor:unique}
\end{corollary}

\subsection{Application to the stochastic APP algorithm}

We aim at studying the iterations of the stochastic APP in terms of random
variables so we consider the argmin set-valued mapping
$M: \omeg \rightrightarrows \espacea{U}$ defined by:
\begin{align}
  M(\omega) = \argmin_{u \in \Uad} K(u) + \proscal{\varepsilon(\va{G}(\omega) + \va{R}(\omega)) - \nabla K (\va{U}(\omega))}{u} + \varepsilon \jad(u, \va{W}(\omega)) \eqfinv
  \label{eq:sto_app_rv}
\end{align}
with $\varepsilon > 0$, $\va{U}(\omega) \in \Uad$,
$\va{W}(\omega) \in \messpace$,
$\va{G}(\omega) \in \partial_{u} \jc(\va{U}(\omega), \va{W}(\omega))$ and
$\va{R}(\omega) \in \ban\dual$. An iteration of the stochastic APP algorithm
consists in solving Problem~\eqref{eq:sto_app_rv}, which is exactly of the form
of Problem~\eqref{eq:aux_pb_indep} with:
\begin{align}
  \va{\varphi}(\omega) = \varepsilon\bp{\va{G}(\omega) + \va{R}(\omega)} - \nabla K \bp{\va{U}(\omega)}\eqfinp
  \label{eq:phi_mir_prox}
\end{align}
In addition to~\ref{hyp:sep_banach}-\ref{hyp:mes_w}, we assume now:
\begin{enumerate}[label=(A\arabic{*}), resume]
\item The function $\jc: \ban \times \messpace \rightarrow \bbR$ that appears in
  Problem~\eqref{eq:master_pb} is jointly \lsc\ and for all $w \in \messpace$,
  $\jc(\cdot, w)$ is proper, convex and subdifferentiable on an open set
  containing $\Uad$. \label{hyp:jc_lsc}
\item The mappings $\va{U}: \omeg \rightarrow \Uad$ and
  $\va{R}: \omeg \rightarrow \ban\dual$ are measurable. \label{hyp:mes_ur}
\end{enumerate}

In~\ref{hyp:mes_ur}, we assume that the mappings $\va{U}$ and $\va{R}$ are
random variables. We cannot do the same for the mapping $\va{G}$ as it must
satisfy $\va{G}(\omega) \in \partial_{u} \jc(\va{U}(\omega), \va{W}(\omega))$
for all $\omega \in \omeg$. In the following proposition, we ensure that there
exists a measurable mapping satisfying this relationship.

\begin{proposition}
  Under Assumptions~\textup{\ref{hyp:sep_banach}, \ref{hyp:complete},
    \ref{hyp:mes_w}, \ref{hyp:jc_lsc}, \ref{hyp:mes_ur}}, the subgradient
  mapping
  $\Gamma : \omega \mapsto \partial_{u} \jc(\va{U}(\omega), \va{W}(\omega))
  \subset \ban\dual$ admits a measurable selection
  $\va{G} : \omeg \rightarrow \ban\dual$.
\end{proposition}

\begin{proof}
  Let
  $
  f(\omega, u) = \jc\bp{u, \va{W}(\omega)}
  $ for $\omega \in \omeg, \ u \in \ban$.
  \begin{itemize}
  \item
    Using that $\ban$ is separable \ref{hyp:sep_banach}, $\va{W}$ is
    measurable~\ref{hyp:mes_w}, $\trib$ is complete~\ref{hyp:complete} and $\jc$
    is \lsc\ \ref{hyp:jc_lsc}, Proposition~\ref{prop:mes_comp} with $h = \jc$
    shows that $f$ is a normal integrand.
  \item We have that for all
    $\omega \in \omeg,\ \Gamma(\omega) = \partial_{u} f\bp{\omega,
      \va{u}(\omega)}$. With~\ref{hyp:jc_lsc}, we get that $f(\omega, \cdot)$ is
    proper for all $\omega \in \omeg$.  We have that $\ban$ and $\ban\dual$ are
    separable~\ref{hyp:sep_banach}, $\va{U}$ is measurable~\ref{hyp:mes_ur} and
    $f$ is a normal integrand, so by~Proposition~\ref{prop:mes_diff}, $\Gamma$
    is Effros-measurable.
  \end{itemize}
  Assumption~\ref{hyp:jc_lsc} ensures that $\Gamma$ is non-empty valued. In
  addition, $\Gamma$ is Effros-measurable and closed-valued in $\ban\dual$ which
  is separable. By Proposition~\ref{prop:cast_rep}, $\Gamma$ admits a measurable
  selection. This means that there exists a measurable function
  $\va{G}: \omeg \rightarrow \ban\dual$ such that for all
  $\omega \in \omeg, \ \va{G}(\omega) \in \Gamma(\omega) = \partial_{u}
  \jc(\va{U}(\omega), \va{W}(\omega))$.
\end{proof}

In the sequel, $\va{G}$ denotes a measurable selection of $\Gamma$.  In order to
apply Theorem~\ref{thm:mes} to prove that the iterates of the stochastic APP
algorithm are measurable, we must ensure that Assumption~\ref{hyp:mes_phi} is
satisfied, that is, we must show that the mapping $\va{\varphi}$ defined
in~\eqref{eq:phi_mir_prox} is measurable. We prove in
Proposition~\ref{prop:phi_mes} that Assumption~\ref{hyp:mes_phi} can be deduced
from the other assumptions.

\begin{proposition}
  Under Assumptions~\textup{\ref{hyp:sep_banach}, \ref{hyp:k}, \ref{hyp:mes_w},
    \ref{hyp:jc_lsc}, \ref{hyp:mes_ur}}, the function $\va{\varphi}$ is
  measurable.\label{item:phi_mes}
  \label{prop:phi_mes}
\end{proposition}

\begin{proof}
  We have already seen in the proof of Theorem~\ref{thm:mes} that
  $\Lambda: (\omega, u) \mapsto K(u)$ is a normal
  integrand. Assumption~\ref{hyp:k} ensures that $\Lambda(\omega, \cdot)$ is
  proper for all $\omega \in \omeg$. We have that $\ban$ and $\ban\dual$ are
  separable~\ref{hyp:sep_banach}, $\va{U}$ is measurable~\ref{hyp:mes_ur}, so
  $\omega \mapsto \nabla_{u} \Lambda\bp{\omega, \va{U}(\omega)} = \nabla
  K\bp{\va{U}(\omega)}$ is measurable by
  Proposition~\ref{prop:mes_diff}. Finally, $\va{R}$ is also
  measurable~\ref{hyp:mes_ur}, so $\va{\varphi}$ is measurable as a sum of
  measurable functions.
\end{proof}

From Theorem~\ref{thm:mes} and Proposition~\ref{prop:phi_mes}, we have obtained
that under Assumptions~\ref{hyp:sep_banach}-\ref{hyp:mes_w},
\ref{hyp:jc_lsc},\ref{hyp:mes_ur}, the mapping $M$ defined
in~\eqref{eq:sto_app_rv} admits a measurable selection. Now, we give the
measurability result for the iterates of the stochastic APP algorithm, which is
defined by the following recursion for $\omega \in \omeg$ and $k \in \bbN$:
\begin{align}
  M_{0}(\omega)
  &= \na{u_{0}} \subset \Uad \eqfinv \label{eq:app_iter}
  \\
  M_{k+1}(\omega)
  &= \argmin_{u \in \Uad} K(u) + \proscal{\varepsilon_{k}\bp{\va{G}_{k}(\omega) + \va{R}_{k}(\omega)} - \nabla K \bp{\va{U}_{k}(\omega)}}{u}\nonumber
  \\
  &\hspace{6cm}+ \varepsilon_{k} \jad\bp{u, \va{W}_{k+1}(\omega)}\eqfinv \nonumber
\end{align}

\begin{theorem}
  Under Assumptions~\textup{\ref{hyp:sep_banach}-\ref{hyp:mes_w},
    \ref{hyp:jc_lsc}, \ref{hyp:mes_ur}}, for all $k \in \bbN$, the mapping
  $M_{k}$ that defines the $k$-th iteration of the stochastic APP
  algorithm~\eqref{eq:app_iter} admits a measurable selection.
  \label{thm:mes_selec_sto_app}
\end{theorem}

\begin{proof}
  The mapping $M_{0}$ admits a measurable selection defined by
  $\va{U}_{0}(\omega) = u_{0}$. Then, by iteratively using the fact
  that~\eqref{eq:sto_app_rv} admits a measurable selection, we deduce that for
  all $k \in \bbN$, $M_{k}$ admits a measurable selection.
\end{proof}

\begin{corollary}
  Assume that~\textup{\ref{hyp:sep_banach}-\ref{hyp:mes_w}, \ref{hyp:jc_lsc},
    \ref{hyp:mes_ur}} are satisfied and that the auxiliary mapping $K$ is
  strongly convex. Then, for all $k \in \bbN$, the unique mapping $\va{U}_{k}$
  that defines the $k$-th iterate of the stochastic APP algorithm is measurable.
  \label{cor:mes_iter}
\end{corollary}

\begin{proof}
  If $K$ is strongly convex, from Corollary~\ref{cor:unique}, we get that
  $M_{k}$ is single-valued, so the iterate $\va{U}_{k}$ is uniquely defined. The
  measurability of $\va{U}_{k}$ follows from
  Theorem~\ref{thm:mes_selec_sto_app}.
\end{proof}

\begin{remark}
  In~\cite[Chapter 14]{rockafellar_variational_2004}, Rockafellar exposes a
  whole set of measurability results in the case where $\ban$ is
  finite-dimensional. The finite-dimensional framework allows to avoid some
  technicalities of the infinite-dimensional case. In particular, the
  completeness assumption~\ref{hyp:complete} is not needed as shown
  by~\cite[Proposition 14.37]{rockafellar_variational_2004} which is the
  analogous of Proposition~\ref{prop:mes_argmin} in the finite-dimensional case.
\end{remark}

\begin{remark}
  In Problem~\eqref{eq:master_pb}, when $\ban$ is a Hilbert space,
  $\Uad = \ban$, $\jad = 0$ {and $\jc$ is assumed to be differentiable
    with respect to $u$}, we can use stochastic gradient descent. Then, we have
  the explicit formula:
  \begin{align}
    \va{U}_{k+1} = \va{U}_{k} - \varepsilon_{k} \nabla_{u} \jc(\va{U}_{k}, \va{W}_{k+1}) \eqfinp
    \label{eq:sgd_va}
  \end{align}
  Under Assumptions~\ref{hyp:sep_banach}, \ref{hyp:mes_w}, \ref{hyp:jc_lsc}, the
  measurability of the iterates is directly obtained by induction using the
  explicit formula~\eqref{eq:sgd_va}.
\end{remark}

\section{Convergence results and efficiency estimates}
\label{sec:conv}

In this section, we prove the convergence of the stochastic APP algorithm for
solving Problem~\eqref{eq:master_pb}.  In addition, we give efficiency estimates
for the convergence of function values. Some technical results for the proofs of
this section are given in the appendix.

\subsection{Convergence of the stochastic APP algorithm}
\label{subsec:conv_app}

Let $\na{\tribu{F}_{k}}_{k \in \bbN}$ be a filtration with
$\tribu{F}_{k} = \sigma\vardelim{\va{W}_{1}, \ldots, \va{W}_{k}}$, where
$\vardelim{\va{W}_{1}, \ldots, \va{W}_{k}}$ are the random variables that appear
in the successive iterations of the stochastic APP algorithm~\eqref{eq:app_iter}
defined on the probability space $(\omeg, \trib, \prbt)$. Recall that,
in~\eqref{eq:app_iter},
$\va{G}_{k} \in \partial_{u} \jc(\va{U}_{k}, \va{W}_{k+1})$~\footnote{In this
  expression, the $\in$ relationship is to be understood \emph{$\omega$ by
    $\omega$}.} is an unbiased stochastic gradient, whereas $\va{R}_{k}$
represents a bias on the gradient.

Convergence results for the stochastic APP algorithm are already proved
in~\cite{culioli_algorithmes_1987,culioli_decomposition/coordination_1990} when
$\ban$ is a Hilbert space (possibly infinite-dimensional) and when there is no
bias $\va{R}_{k}$.  In~\cite{geiersbach_projected_2019}, convergence of the
projected stochastic gradient descent is proved in a Hilbert space and with a
bias $\va{R}_{k}$. For stochastic mirror descent {in the
  finite-dimensional setting}, convergence results and efficiency estimates can
be found in~\cite{nemirovski_robust_2009}, but no bias is considered. Here, we
present convergence results for the stochastic APP algorithm in a
  reflexive {separable} Banach space and we allow for a bias $\va{R}_{k}$, hence
generalizing previous results.

In addition to~\ref{hyp:sep_banach}-\ref{hyp:mes_w}, \ref{hyp:jc_lsc},
\ref{hyp:mes_ur}, we make the following assumptions:
\begin{enumerate}[resume, label=(A\arabic*)]
\item The functions $\jc(\cdot, w): \ban \rightarrow \bbR$ and
  $\jad(\cdot, w): \ban \rightarrow \bbR$ have linearly bounded subgradient in
  $u$, uniformly in $w \in \messpace$:
  \begin{align*}
    \left\{
    \begin{aligned}
        &\exists c_{1}, c_{2} > 0 \eqsepv
        \forall (u, w) \in \Uad \times \messpace \eqsepv
        \forall r \in \partial_{u} \jc(u,w) \eqsepv
        \nnorm{r} \leq c_{1} \nnorm{u} + c_{2} \eqfinp\\
        &\exists d_{1}, d_{2} > 0 \eqsepv
        \forall (u, w) \in \Uad \times \messpace \eqsepv
        \forall s \in \partial_{u} \jad(u,w) \eqsepv
        \nnorm{s} \leq d_{1} \nnorm{u} + d_{2} \eqfinp
      \end{aligned}
    \right.
  \end{align*}
  \label{hyp:sglb}
\item The objective function $\J$ is coercive on $\Uad$. \label{hyp:coer}
\item The function $K$ is $b$-strongly convex for $b > 0$, meaning that
  for all $u,v \in \espacea{U}$:
  \begin{align*}
      K(v) \geq K(u) + \bscal{\nabla K(u)}{v-u} + \frac{b}{2} \sqnorm{u-v} \eqfinv
  \end{align*}
  and $\nabla K$ is $L_{K}$-Lipschitz continuous with $L_{K} > 0$, that is, for all $u, v \in \ban$:
  \begin{align*}
    \normdual{\nabla K(v) - \nabla K(u) } \leq L_{K} \norm{v - u} \eqfinv
  \end{align*}
  where $\normdual{\cdot}$ is the dual norm on $\ban\dual$. \label{hyp:aux_func}
\item \label{hyp:sig_seq} The sequence of step sizes
  $\na{\varepsilon_{k}}_{k \in \bbN}$, {with $\varepsilon_{k} > 0$ for
    all $k$}, satisfies $\sum_{k} \varepsilon_{k} = + \infty$ and
  $\sum_{k} \varepsilon_{k}^{2} < + \infty$.
\item\label{hyp:err_sum} Each $\va{R}_{k}$ is measurable with respect to
  $\tribu{F}_{k+1}$, the sequence of random variables
  $\ba{\va{R}_{k}}_{k \in \bbN}$ is $\prbt$-almost surely (\Pas)
  bounded,\footnote{The set
    $\ba{\omega \in \omeg, \ \na{\va{R}_{k}(\omega)}_{k \in \bbN} \ \text{is
        unbounded}}$ is negligible.} and we have:
  \begin{align*}
    \sum_{k \in \bbN} \varepsilon_{k} \bespc{\norm{\va{R}_{k}}}{\tribu{F}_{k}} < + \infty \quad \Pas
  \end{align*}
\item\label{hyp:quasi-integrable} { For all integers $k\ge 1$, the
    integrand $\hkintegrand:\ban \times \Omega \rightarrow \bbR$ defined by
    $\hkintegrand(u,\omega)=(\jc+\jad)(u,\va{W}_{k}(\omega))$ for all
    $(u,\omega)\in \ban{\times}\Omega$ is
    $\trib$-quasi-integrable~\cite{thibault}. That is, for each $k\ge 1$, there
    exists an integrable mapping $\psi_k:\Omega \to {\bbR}$ such that
    $\psi_k \le \hkintegrand(u,\cdot)$ for all $u\in \ban$.}
\end{enumerate}

{
We make some comments on Assumptions~\ref{hyp:sglb}-\ref{hyp:quasi-integrable}:
\begin{itemize}
  \item Assumption~\ref{hyp:sglb} is a relaxation of the standard assumption of bounded gradients, used in~\cite{nemirovski_robust_2009} for example.
  \item Assumption~\ref{hyp:coer} is used to ensure the existence of solutions to Problem~\eqref{eq:master_pb} when $\Uad$ is an unbounded domain. When $\Uad$ is bounded, Assumption~\ref{hyp:coer} is automatically satisfied.
  \item Assumption~\ref{hyp:aux_func} is related to the user-defined function $K$ and not to the intrinsic characteristics of the minimization problem, thus it is not restrictive.
  \item Assumptions~\ref{hyp:sig_seq} and~\ref{hyp:err_sum} are standard to ensure the convergence of SA schemes. In particular, \ref{hyp:err_sum} ensures that the noise on the gradient vanishes sufficiently fast.
  \item Assumption~\ref{hyp:quasi-integrable} is a technical assumption used to ensure that the conditional expectation of the integrand $\hkintegrand$ is defined. As a sufficient condition, assuming that $\jc+\jad$ is nonnegative would ensure
    Assumption~\ref{hyp:quasi-integrable}.
\end{itemize}
}

Assumptions~\ref{hyp:sep_banach}-\ref{hyp:jad}, \ref{hyp:jc_lsc} and \ref{hyp:coer}
ensure that $J$ is well-defined, convex, \lsc, coercive and attains its minimum
on $\Uad$. Hence, Problem~\eqref{eq:master_pb} has a non-empty set of solutions
$U\opt$. From now
on, $K$ is supposed to be $b$-strongly convex, so by
Corollary~\ref{cor:mes_iter}, the problem solved at each iteration~$k$ of the
stochastic APP algorithm admits a unique solution $\va{U}_{k+1}$, which is
measurable.

We start by a technical lemma which gives a key inequality
that will be used for the proof of convergence of the stochastic APP algorithm in Theorem~\ref{thm:cv} and to derive efficiency estimates in Theorems~\ref{thm:cv_rate_expec_avrg} and~\ref{thm:cv_rate_expec_last}.

\begin{lemma}
  \label{lem:lyap_bd}
  Let $v \in \Uad$ and consider the Lyapunov function:
  \begin{equation}
    \label{def:lyap}
    \ell_{v}(u) = K(v) - K(u) - \bscal{\nabla K(u)}{v-u}
    \eqsepv u \in \Uad \eqfinp
  \end{equation}
  Let $\na{u_{k}}_{k \in \bbN}$ be the sequence of iterates of
  Algorithm~\ref{alg:sto_ppa} corresponding to the realization
  $\na{w_{k}}_{k \in \bbN}$ of the stochastic process
  $\ba{\va{W}_{k}}_{k \in \bbN}$. Then, under
  Assumptions~\textup{\ref{hyp:jc_lsc}, \ref{hyp:sglb}} and
  \textup{\ref{hyp:aux_func}}, there exist constants
  $\alpha, \beta, \gamma, \delta > 0$ such that, for all $k \in \bbN$:
  \begin{multline}
    \ell_{v}(u_{k+1}) \leq
    \Bp{1+\alpha\varepsilon_{k}^{2}
         + \frac{2}{b}\varepsilon_{k}\nnorm{r_{k}}} \ell_{v}(u_{k})
         + \beta\varepsilon_{k}^{2}\ell_{v}(u_{k+1}) \\
    + \Bp{\gamma\varepsilon_{k}^{2}+\varepsilon_{k}\nnorm{r_{k}}
        +\delta\np{\varepsilon_{k}\nnorm{r_{k}}}^{2}} \\
    + \varepsilon_{k} \bp{(\jc+\jad)(v,w_{k+1})-(\jc+\jad)(u_{k},w_{k+1})} \eqfinv
    \label{eq:bound_real}
  \end{multline}
  where we recall that $b > 0$ is the strong convexity constant of $K$,
  $\varepsilon_{k}$ is the step size and $r_{k}$ is the additive error on the
  stochastic gradient at iteration $k$ of the algorithm.
\end{lemma}

\begin{proof}
  By~\ref{hyp:aux_func}, $K$ is $b$-strongly convex implying that:
  \begin{align}
    \frac{b}{2} \nnorm{u-v}^{2} \leq \ell_{v}(u) \eqfinp
    \label{eq:stoch-pbsansc-ppa-lyap}
  \end{align}
  This shows that $\ell_{v}$ is lower bounded and coercive. Let $k \in \bbN$, as
  $u_{k+1}$ is solution of~\eqref{pb:stoch-pbsansc-ppa}, it solves the following
  variational inequality, {characterizing the minimum of the sum of a
    Gateaux-differentiable and a non-differentiable function~\cite[Chapter~II,
    Proposition 2.2]{ekeland_convex_1976}}: for all $u \in \Uad$,
    \begin{multline}
      \label{eq:stoch-pbsansc-ppa-co}
      \bscal{\nabla K(u_{k+1})-\nabla K(u_{k})+\varepsilon_{k}(g_{k}
      +r_{k})}{u-u_{k+1}}
      \\
      + \varepsilon_{k}( \jad(u,w_{k+1})-\jad(u_{k+1},w_{k+1}))
      \geq 0 \eqfinp
    \end{multline}
    We have:
    \begin{multline}
      \ell_{v}(u_{k+1}) - \ell_{v}(u_{k}) =
         \underbrace{K(u_{k})-K(u_{k+1}) -
                     \bscal{\gradi{K}(u_{k})}{u_{k}-u_{k+1}}}_{T_{1}} \\
         + \underbrace{\bscal{\gradi{K}(u_{k}) -
                     \gradi{K}(u_{k+1})}{v-u_{k+1}}}_{T_{2}} \eqfinp
    \end{multline}
      As $K$ is convex~\ref{hyp:aux_func}, we get
        $T_{1} \leq 0.$
      The optimality condition~\eqref{eq:stoch-pbsansc-ppa-co}
      at~$u = v$ implies:
      \begin{align*}
        T_{2} & \leq \varepsilon_{k} \bscal{g_{k}+r_{k}}{v-u_{k+1}}
        + \varepsilon_{k} \bp{\jad(v,w_{k+1})-\jad(u_{k+1},w_{k+1})} \\
        & \leq \varepsilon_{k} \Bp{
          \underbrace{\bscal{g_{k}}{v-u_{k}}+
          \jad(v,w_{k+1})-\jad(u_{k},w_{k+1})}_{T_{3}}
          + \underbrace{\bscal{r_{k}}{v-u_{k}}}_{T_{4}} \\
        & \hspace{2cm}
          + \underbrace{\bscal{g_{k}+r_{k}}{u_{k}-u_{k+1}}+
           \jad(u_{k},w_{k+1})-\jad(u_{k+1},w_{k+1})}_{T_{5}} } \eqfinp
      \end{align*}
      \begin{itemize}
        \item As $\jc(\cdot,w_{k+1})$ is convex~\ref{hyp:jc_lsc}, we get:
        \begin{align*}
          T_{3} \leq \bp{\jc+\jad}(v,w_{k+1}) - \bp{\jc+\jad}(u_{k},w_{k+1}) \eqfinp
        \end{align*}
        \item By Cauchy-Schwarz inequality, using~$a\leq a^2 +1$ for~$a \geq 0$ and~\eqref{eq:stoch-pbsansc-ppa-lyap}, we get:
        \begin{align*}
          T_{4}
          & \leq \nnorm{r_{k}} \nnorm{v-u_{k}} \leq \nnorm{r_{k}} \bp{\nnorm{v-u_{k}}^{2}+1} \leq \nnorm{r_{k}} + \frac{2}{b} \ell_{v}(u_{k}) \nnorm{r_{k}} \eqfinp
        \end{align*}
      \item The optimality condition~\eqref{eq:stoch-pbsansc-ppa-co}
        at~$u = u_{k}$ and the strong monotonicity of~$\nabla K$, that arises
        from~\ref{hyp:aux_func}, imply:
        \begin{align}
          \begin{aligned}
          \label{eq:co-uk}
          b \bnorm{u_{k+1}-u_{k}}^{2}
          &\leq
          \varepsilon_{k} \bp{\nscal{g_{k}+r_{k}}{u_{k}-u_{k+1}}
          \\
          &\hspace{1cm} + \jad(u_{k},w_{k+1})-\jad(u_{k+1},w_{k+1})} \eqfinv
        \end{aligned}
        \end{align}
        where we recognize~{$\varepsilon_{k}T_{5}$} as the right-hand
        side.  Using the linearly bounded subgradient property
        of~$\jad$~\ref{hyp:sglb} with the result of
        Proposition~\ref{prop:sglb-lip}, we get:
        \begin{align*}
          \babs{\jad(u_{k},w_{k+1})&-\jad(u_{k+1},w_{k+1})}
          \\
          & \leq \Bp{d_{1}\max\ba{\nnorm{u_{k}},\nnorm{u_{k+1}}}+d_{2}}
            \nnorm{u_{k}-u_{k+1}} \eqfinv \\
          & \leq \Bp{d_{1}\bp{\nnorm{u_{k}}+\nnorm{u_{k+1}}}+d_{2}}
            \nnorm{u_{k}-u_{k+1}} \eqfinp
        \end{align*}
        With Cauchy-Schwarz inequality on the first term of~$T_{5}$, we have:
        \begin{align*}
          T_{5} \leq
          \nnorm{g_{k}+r_{k}}\nnorm{u_{k}-u_{k+1}} +
          \bp{d_{1}\nnorm{u_{k}}+d_{1}\nnorm{u_{k+1}}+d_{2}}
          \nnorm{u_{k}-u_{k+1}} \eqfinp
        \end{align*}
        By the triangular inequality and Assumption~\ref{hyp:sglb} for~$\jc$, we
        deduce that there exist positive constants~$e_{1}$, $e_{2}$ and $e_{3}$
        such that:
        \begin{align*}
          T_{5} \leq
          \bp{e_{1}\nnorm{u_{k}}+e_{2}\nnorm{u_{k+1}}+e_{3}+\nnorm{r_{k}}}
          \nnorm{u_{k+1}-u_{k}} \eqfinp
        \end{align*}
        By Inequality~\eqref{eq:co-uk}, we then get:
        \begin{align}
        \label{eq:stoch-pbsansc-ppa-maj-uk}
          \bnorm{u_{k+1}-u_{k}} \leq \frac{\varepsilon_{k}}{b}
          \bp{e_{1}\nnorm{u_{k}}+e_{2}\nnorm{u_{k+1}}+e_{3}+\nnorm{r_{k}}} \eqfinv
        \end{align}
        and therefore by a repeated use of~$(a+b)^2 \leq 2 (a^{2}+b^{2})$,
        \begin{align*}
          T_{5}
          & \leq \frac{4\varepsilon_{k}}{b}
            \bp{e_{1}^{2}\sqnorm{u_{k}}+e_{2}^{2}\sqnorm{u_{k+1}}+
            e_{3}^{2}+\sqnorm{r_{k}}} \eqfinp
        \end{align*}
        We bound~$\nnorm{u_{k}}$ (resp.~$\nnorm{u_{k+1}}$) by
        $\nnorm{u_{k}-v}+\nnorm{v}$ (resp. $\nnorm{u_{k+1}-v}+\nnorm{v}$) and use~\eqref{eq:stoch-pbsansc-ppa-lyap} to
        infer that there exist positive constants
        $\alpha$, $\beta$, $\gamma$, $\delta$ such that:
        \begin{align*}
          T_{5} \leq \varepsilon_{k}
          \bp{\alpha\ell_{v}(u_{k})+\beta\ell_{v}(u_{k+1})+
          \gamma+\delta\nnorm{r_{k}}^{2}} \eqfinp
        \end{align*}
      \end{itemize}
    Collecting the bounds for $T_{1}, T_{3}, T_{4}$ and $T_{5}$, we get the desired result.
\end{proof}

When no bias is present, $r_{k} = 0$, we retrieve the same inequality as
in~\cite[\S 2.5.1]{culioli_algorithmes_1987}. In the proofs of the subsequent
theorems, Inequality~\eqref{eq:bound_real} is fundamental to derive boundedness
properties or convergence results for the Lyapunov function $\ell_{v}$.

We give convergence results for the stochastic APP algorithm, in terms of
function values as well as for the iterates. The proof is similar to that
in~\cite{culioli_algorithmes_1987,culioli_decomposition/coordination_1990} (case
of a Hilbert space, no bias considered). The assumption that the Banach $\ban$
is reflexive~\ref{hyp:sep_banach} allows for a similar treatment as in the
Hilbert case. The additional contribution of the bias is already taken care of
by Inequality~\eqref{eq:bound_real}. {We denote by $J\opt$ the value of
  $\J$ on the non-empty set of solutions $U\opt$ of
  Problem~\eqref{eq:master_pb}.}

\begin{theorem}
  Under Assumptions~\textup{\ref{hyp:sep_banach}-\ref{hyp:mes_w}, \ref{hyp:jc_lsc}-\ref{hyp:quasi-integrable}}, we have the following statements:
  \begin{itemize}
    \item The sequence of random variables $\ba{\J(\va{U}_{k})}_{k \in \bbN}$ converges to $J\opt$ almost surely.
    \item The sequence of iterates $\ba{\va{U}_{k}}_{k \in \bbN}$ of the stochastic APP algorithm is almost surely bounded and every weak cluster point of a bounded realization of this sequence belongs to the optimal set $U\opt$.
  \end{itemize}
  \label{thm:cv}
\end{theorem}

\begin{proof}
  Let $u\opt \in U\opt$ be a solution of Problem~\eqref{eq:master_pb}.

  \begin{enumerate}[wide]
  \item \textbf{Upper bound on the variation of the Lyapunov function.}
    We write the inequality of Lemma~\ref{lem:lyap_bd} at $v = u\opt$ in terms of random variables {and reorganize the terms
    as follows:
      \begin{multline}
        \label{eq:inequality-ps}
        \underbrace{ (1-\beta\varepsilon_{k}^{2}) \ell_{u\opt}(\va{u}_{k+1})}_{\va{A}_{k+1}} \leq
        \underbrace{\Bp{1+\alpha\varepsilon_{k}^{2}
            + \frac{2}{b}\varepsilon_{k}\nnorm{\va{r}_{k}}} \ell_{u\opt}(\va{u}_{k})}_{\va{B}_k}
        \\
        + \underbrace{\Bp{\gamma\varepsilon_{k}^{2}+\varepsilon_{k}\nnorm{\va{r}_{k}}
          +\delta\np{\varepsilon_{k}\nnorm{\va{r}_{k}}}^{2}}}_{\va{C}_k} \\
      - \underbrace{\varepsilon_{k} \bp{(\jc+\jad)(\va{u}_{k},\va{w}_{k+1})-(\jc+\jad)(u\opt,\va{w}_{k+1})}}_{\va{D}_k} \eqfinv
    \end{multline}
      this last inequality being valid $\Pas$}.  { We assume without loss
      of generality that $(1-\beta\varepsilon_{k}^{2}) >0$ as it is true for $k$
      large enough.  As $\va{D}_k$ takes only finite values since $\jc+\jad$
      takes finite values, we obtain from Equation~\eqref{eq:inequality-ps} that
      almost surely $\va{A}_{k+1} + \va{D}_k \le \va{B}_k + \va{C}_k$.  It is
      classical to define extended conditional expectation for nonnegative
      random variables or more generally for $\trib$-quasi-integrable random
      variables~\cite[p. 339]{thibault}. Thus the (extended) conditional
      expectation with respect to $\tribu{F}_{k}$ is well defined for each of
      the three terms $\va{A}_{k+1}$, $\va{B}_{k}$, $\va{C}_{k}$.  Moreover, as
      $\va{B}_k$ and $\va{C}_k$ are both nonnegative, we have that
      $\bespc{\va{B}_k+\va{C}_k}{\tribu{F}_{k}}=
      \bespc{\va{B}_k}{\tribu{F}_{k}}+ \bespc{\va{C}_k}{\tribu{F}_{k}}$.  }
    We now prove that the conditional expectation of $\va{D}_k$ with
      respect to $\tribu{F}_{k}$ exists and satisfies
      $\bespc{ \va{D}_k}{\tribu{F}_{k}} =\varepsilon_{k} ( \J(\va{u}_k)-
      \J(u\opt))$.  For that purpose, we consider the integrand
      $\hkintegrand: \ban \times \Omega \rightarrow \bbR$ defined by
      $\hkintegrand(u,\omega)=(\jc+\jad)(u,\va{W}_{k+1}(\omega))$ for all
      $(u,\omega)\in \ban{\times}\Omega$ and recall some results
      from~\cite{thibault}. The integrand $\hkintegrand$ is
      $\trib$-quasi-integrable (by Assumption~\ref{hyp:quasi-integrable}) and
      \lsc\ (using the assumptions on $\jc$ and $\jad$). Therefore, there
      exists an $\trib$-quasi-integrable integrand $\hkintegrand^{\tribu{F}_k}$
      which gives the conditional expectation with respect to $\tribu{F}_{k}$ of
      the integrand $\hkintegrand$, that is
      $\hkintegrand^{\tribu{F}_k}(u,\omega) =
      \bespc{\hkintegrand(u,\cdot)}{\tribu{F}_{k}}(\omega)$ for all
      $(u,\omega)\in \ban{\times}\Omega$~\cite[Proposition~12]{thibault}.
      {We have}
    \begin{align*}
      \hkintegrand^{\tribu{F}_k}(u,\omega)
      &= \bespc{\hkintegrand(u,\cdot)}{\tribu{F}_{k}}(\omega)
        = \bespc{(\jc+\jad)(u,\va{W}_{k+1})}{\tribu{F}_{k}}(\omega)
      \eqfinv \\
      \intertext{{where the conditional expectation is an expectation since
      $\va{W}_{k+1}$ is independent of the $\sigma$-field $\tribu{F}_k$,
      so that}}
      \hkintegrand^{\tribu{F}_k}(u,\omega)
      &= \besp{(\jc+\jad)(u,\va{W}_{k+1})} = \J(u)
        \eqfinv
    \end{align*}
    {that is, $\hkintegrand^{\tribu{F}_k}$ does not depend on~$\omega$.}
    Moreover, given any $\tribu{F}_k$-measurable random variable
    $\va{Y}$ we have~\cite[Proposition~13]{thibault} that
    \begin{equation*}
      \hkintegrand^{\tribu{F}_k}(\va{Y}(\omega),\omega)= \bespc{\hkintegrand(\va{Y}(\cdot),\cdot)}{\tribu{F}_{k}}(\omega)
      \eqfinv
    \end{equation*}
    which, using the fact that $\va{u}_{k}$ is
    $\tribu{F}_k$-measurable, gives for all $\omega\in \Omega$ that
    \begin{align*}
      \varepsilon_{k}\bp{   \J(\va{u}_k(\omega))- \J(u\opt)}
      &=
        \varepsilon_{k}  \bp{ \hkintegrand^{\tribu{F}_k}(\va{U}_k(\omega),\omega) -  \hkintegrand^{\tribu{F}_k}(u\opt,\omega)}
      \\
      & =\varepsilon_{k}
        \bp{ \bespc{\hkintegrand(\va{U}_k(\cdot),\cdot)}{\tribu{F}_{k}}(\omega)
        - \bespc{\hkintegrand(u\opt,\cdot)}{\tribu{F}_{k}}(\omega)}
      \\
      & =\varepsilon_{k}
        \bespc{\hkintegrand(\va{U}_k(\cdot),\cdot)
        - \hkintegrand(u\opt,\cdot)}{\tribu{F}_{k}}(\omega)
        \tag{as $\bespc{\hkintegrand(u\opt,\cdot)}{\tribu{F}_{k}}=\J(u\opt)$ is finite}
      \\
      & =
        \varepsilon_{k} \bespc{  (\jc+\jad)(\va{u}_{k},\va{w}_{k+1})-(\jc+\jad)(u\opt,\va{w}_{k+1})}{\tribu{F}_{k}}(\omega)
      \\
      & = \bespc{ \va{D}_k}{\tribu{F}_{k}}(\omega)
        \nonumber \eqfinp
    \end{align*}
    Thus, the conditional expectations with respect to $\tribu{F}_{k}$ of $\va{A}_{k+1}$, $\va{B}_k$, $\va{C}_k$ and $\va{D}_k$
    are properly defined.
    To conclude it remains to show that $\bespc{\va{A}_{k+1}+\va{D}_k}{\tribu{F}_{k}}=
    \bespc{\va{A}_{k+1}}{\tribu{F}_{k}} + \bespc{\va{D}_k}{\tribu{F}_{k}}$ which is the case as
    $\bespc{\va{D}_k}{\tribu{F}_{k}}=\varepsilon_{k}\bp{   \J(\va{u}_k)- \J(u\opt)}$ is almost surely
    finite and nonnegative.
    Then we can take the conditional expectation on both sides of $\va{A}_{k+1} + \va{D}_k \le \va{B}_k + \va{C}_k$ with respect to the
    $\sigma$-field $\tribu{F}_{k}$ and use the just proved additivity properties to obtain
    \begin{equation}
      \label{eq:stoch-pbsansc-ppa-maj-psik-step}
      (1-\va{\beta}_{k}) \bespc{\ell_{u\opt}\bp{\va{u}_{k+1}}}{\tribu{F}_{k}}
      + \varepsilon_{k} \bp{\J(\va{u}_{k}) - \J(u\opt)}
      \leq
      \np{1+\va{\alpha}_{k}} \ell_{u\opt}\bp{\va{u}_{k}}
      + \va{\gamma}_{k}
      \eqfinv
    \end{equation}
    where we have:
    \begin{align*}
      \va{\alpha}_{k}
      & = \alpha\varepsilon_{k}^{2} +
        \frac{2}{b}\varepsilon_{k}\bespc{\nnorm{\va{r}_{k}}}{\tribu{F}_{k}} \eqfinv \\
      \va{\beta}_{k}
      & = \beta\varepsilon_{k}^{2} \eqfinv \\
      \va{\gamma}_{k}
      & = \gamma\varepsilon_{k}^{2} +
        \varepsilon_{k}\bespc{\nnorm{\va{r}_{k}}}{\tribu{F}_{k}} +
        \delta\vardelim{\varepsilon_{k}\bespc{\nnorm{\va{r}_{k}}}{\tribu{F}_{k}}}^{2}
        \eqfinv
    \end{align*}
    and where we have used that $\va{U}_{k}$ is
    $\tribu{F}_{k}$-measurable, to obtain that
    $\bespc{\ell_{u\opt}\bp{\va{u}_{k}}}{\tribu{F}_{k}}=\ell_{u\opt}\bp{\va{u}_{k}}$.
  By Assumptions~\ref{hyp:sig_seq} and~\ref{hyp:err_sum},
  $\va{\alpha}_{k}, \va{\beta}_{k}$ and $\va{\gamma}_{k}$ are the terms of
  convergent series. Recall that $\J(\va{u}_{k})-\J(u\opt)$ is almost surely
  nonnegative as $u\opt$ is solution of~\eqref{eq:master_pb}. The right hand side of
  Equation~\eqref{eq:stoch-pbsansc-ppa-maj-psik-step} is almost surely finite and since the left hand side is
  the sum of two positive terms, each of them is almost surely finite. Thus we also have almost surely that
  \begin{multline}
    \label{eq:stoch-pbsansc-ppa-maj-psik}
    \bespc{\ell_{u\opt}\bp{\va{u}_{k+1}}}{\tribu{F}_{k}}
    \leq
    \np{1+\va{\alpha}_{k}} \ell_{u\opt}\bp{\va{u}_{k}}
    + \va{\beta}_{k}\bespc{\ell_{u\opt}\bp{\va{u}_{k+1}}}{\tribu{F}_{k}} \\
    + \va{\gamma}_{k}
    - \varepsilon_{k} \bp{\J(\va{u}_{k}) - \J(u\opt)}
    \eqfinp
  \end{multline}
  \item \textbf{Convergence analysis.} \label{cv_lyap} Applying Corollary~\ref{cor:robbins-siegmund} of Robbins-Siegmund theorem, we get that the sequence of random variables $\ba{\ell_{u\opt}\bp{\va{u}_{k}}}_{k\in\mathbb{N}}$ converges $\Pas$ to a finite random variable~$\va{\ell}_{u\opt}^{\infty}$
    and we have:
    \begin{align}
      \label{eq:stoch-pbsansc-ppa-sum}
      \sum_{k=0}^{+\infty} \varepsilon_{k}
      \bp{\J(\va{u}_{k})-\J(u\opt)} < + \infty \;\; \Pas \eqfinp
    \end{align}
  \item \label{last-item} \textbf{Limits of sequences.} The sequence
    $\ba{\ell_{u\opt}( \va{u}_{k})}_{k \in \bbN}$ is $\Pas$ bounded, so
    by~\eqref{eq:stoch-pbsansc-ppa-lyap}, we get that the sequence
    $\ba{\va{U}_{k}}_{k \in \bbN}$ is also $\Pas$
    bounded. Assumption~\ref{hyp:sglb} then implies that the sequence
    $\ba{\va{G}_{k}}_{k \in \bbN}$ is also $\Pas$ bounded. Finally, as the
    sequence $\ba{\va{R}_{k}}_{k \in \bbN}$ is assumed to be $\Pas$
    bounded~\ref{hyp:err_sum}, we deduce
    from~\eqref{eq:stoch-pbsansc-ppa-maj-uk} that the sequence
    $\ba{\nnorm{\va{u}_{k+1}-\va{u}_{k}}/\varepsilon_{k}}_{k\in\bbN}$ is also
    $\Pas$ bounded. This last property ensures that Assumption~(c) of
    Proposition~\ref{prop:tech2}, is satisfied. Assumption~(b) of
    Proposition~\ref{prop:tech2} is exactly~\eqref{eq:stoch-pbsansc-ppa-sum} and
    Assumption~(a) is satisfied as we have~\ref{hyp:sig_seq}. On a bounded set
    containing the sequence $\ba{\va{U}_{k}}_{k \in \bbN}$, for instance the
    convex hull of this sequence, the function $\J$ is Lipschitz continuous by
    Proposition~\ref{prop:sglb-lip}. This ensures the continuity assumption
    required to apply Proposition~\ref{prop:tech2}. We conclude that
    $\ba{\J(\va{U}_{k})}_{k \in \bbN}$ converges almost surely to
    $\J(u\opt) = J\opt$, the optimal value of Problem~\eqref{eq:master_pb}.

    Let $\omeg_{0}$ be the negligible subset of $\omeg$ on which the sequence
    $\ba{\ell_{u\opt}\bp{\va{u}_{k}}}_{k \in \bbN}$ is unbounded and $\omeg_{1}$
    be the negligible subset of $\omeg$ on which the
    relation~\eqref{eq:stoch-pbsansc-ppa-sum} is not satisfied. We have
    $\prbt \bp{\omeg_{0}\cup\omeg_{1}}=0$. Let
    $\omega \notin \omeg_{0} \cup \omeg_{1}$. The sequence
    $\na{u_{k}}_{k \in \bbN}$ associated to this element $\omega$ is bounded and
    each $u_{k}$ is in $\Uad$, a closed subset of $\ban$.  As $\ban$ is
    reflexive~\ref{hyp:sep_banach}, there exists a weakly converging subsequence
    $\{u_{\xi(k)}\}_{k \in \bbN}$. Note that $\na{\xi(k)}_{k \in \bbN}$ depends
    on $\omega$.  Let $\overline{u}$ be the weak limit of the sequence
    $\na{u_{\xi(k)}}_{k \in \bbN}$. The function $\J$ is \lsc\ and convex, it
    is then weakly \lsc\ by~\cite[Corollary 2.2]{ekeland_convex_1976}. Thus we
    have:
    \begin{align*}
      \J(\overline{u}) \leq
      \liminf_{k \rightarrow +\infty} \J(u_{\xi(k)}) = \J(u\opt) \eqfinp
    \end{align*}
  \end{enumerate}
  We conclude that $\overline{u} \in U\opt$.
\end{proof}

When the differential of $K$ is weakly continuous, we can prove stronger
convergence results for the sequence of iterates of the stochastic APP
algorithm. These results already appear in~\cite{culioli_algorithmes_1987} and
remain valid for our more general version of the algorithm.

\begin{theorem}
  Consider again~\textup{\ref{hyp:sep_banach}-\ref{hyp:mes_w},
    \ref{hyp:jc_lsc}-\ref{hyp:err_sum}} and suppose that the differential of $K$
  is weakly continuous. Then, the sequence of iterates $\ba{\va{U}_{k}}$
  converges weakly $\Pas$ to a single element of $U\opt$. If moreover, the
  function $\Jc$ is strongly convex, then, the sequence of iterates
  $\ba{\va{U}_{k}}$ converges strongly $\Pas$ to the unique solution $u\opt$ of
  Problem~\eqref{eq:master_pb}.
\end{theorem}

\begin{proof}
  Consider the case where the differential of $K$ is weakly continuous. Let
  $\na{u_{k}}_{k \in \bbN}$ be a sequence generated by the algorithm. Suppose
  that there exist two subsequences~$\{u_{\xi(k)}\}_{k\in\bbN}$
  and~$\{u_{\psi(k)}\}_{k\in\bbN}$ converging weakly respectively to two
  solutions~$\overline{u}_{\xi}$ and~$\overline{u}_{\psi}$ of the problem,
  with~$\overline{u}_{\xi} \neq \overline{u}_{\psi}$. Then we have:
  \begin{multline}
    K(\overline{u}_{\psi})-K(u_{\xi(k)})-
    \bscal{\nabla K(u_{\xi(k)})}{\overline{u}_{\psi}-u_{\xi(k)}} =
    K(\overline{u}_{\psi})-K(\overline{u}_{\xi})
    -\bscal{\nabla K(u_{\xi(k)})}{\overline{u}_{\psi}-\overline{u}_{\xi}} \\
    + \bp{K(\overline{u}_{\xi})-K(u_{\xi(k)})-
          \bscal{\nabla K(u_{\xi(k)})}{\overline{u}_{\xi}-u_{\xi(k)}}} \eqfinp
  \end{multline}
  By the point~\ref{cv_lyap} of the proof of Theorem~\ref{thm:cv},
  \begin{align*}
    & \lim_{k\rightarrow+\infty} K(\overline{u}_{\psi})-K(u_{\xi(k)})-
      \bscal{\nabla K(u_{\xi(k)})}{\overline{u}_{\psi}-u_{\xi(k)}} =
      \lim_{k\rightarrow+\infty} \ell_{\overline{u}_{\psi}}(u_{k}) =
      \ell_{\overline{u}_{\psi}} \eqfinv \\
    & \lim_{k\rightarrow+\infty} K(\overline{u}_{\xi})-K(u_{\xi(k)})-
      \bscal{\nabla K(u_{\xi(k)})}{\overline{u}_{\xi}-u_{\xi(k)}} =
      \lim_{k\rightarrow+\infty} \ell_{\overline{u}_{\xi}}(u_{k}) =
      \ell_{\overline{u}_{\xi}} \eqfinv
  \end{align*}
  therefore by weak continuity of $\nabla K$ and strong convexity of~$K$, we get:
  \begin{align*}
    \ell_{\overline{u}_{\psi}} - \ell_{\overline{u}_{\xi}}
     & = K(\overline{u}_{\psi})-K(\overline{u}_{\xi}) -
        \bscal{\nabla K(\overline{u}_{\xi})}{\overline{u}_{\psi}-\overline{u}_{\xi}}
      \geq \frac{b}{2}\nnorm{\overline{u}_{\xi}-\overline{u}_{\psi}}^{2} \eqfinp
  \end{align*}
  Inverting the roles of $\overline{u}_{\psi}$ and~$\overline{u}_{\xi}$, by a similar calculation as previously we get:
  \begin{align*}
    \ell_{\overline{u}_{\xi}} - \ell_{\overline{u}_{\psi}}
    \geq \frac{b}{2}\nnorm{\overline{u}_{\xi}-\overline{u}_{\psi}}^{2} \eqfinv
  \end{align*}
  We then deduce that ~$\overline{u}_{\xi}=\overline{u}_{\psi}$, which
  contradicts the initial assumption. We conclude that all weakly converging
  subsequences of the sequence $\na{u_{k}}$ converge to the same limit, hence we
  have the weak convergence of the whole sequence $\na{u_{k}}$ to a single
  element of $U\opt$.

  Now consider the case where $\Jc$ is strongly convex, with constant $a$. Then,
  Problem~\eqref{eq:master_pb} admits a unique solution $u\opt$ which is
  characterized by the following variational inequality:
  \begin{align*}
    \exists r\opt \in \partial \Jc(u\opt) \eqsepv \forall u \in \Uad \eqsepv
    \bscal{r\opt}{u-u\opt} + \Jad(u) - \Jad(u\opt) \geq 0 \eqfinp
  \end{align*}
  The strong convexity assumption on $\Jc$ yields:
  \begin{align*}
    \J(\va{u}_{k}) - \J(u\opt)
    & \geq \bscal{r\opt}{\va{u}_{k}-u\opt}
           + \frac{a}{2} \nnorm{\va{u}_{k}-u\opt}^{2}
            + \Jad(\va{u}_{k}) - \Jad(u\opt)
    \geq \frac{a}{2} \nnorm{\va{u}_{k}-u\opt}^{2} \eqfinp
  \end{align*}
  As~$\ba{\J(\va{u}_{k})}_{k \in \bbN}$ converges almost surely to~$\J(u\opt)$,
  we get that~$\nnorm{\va{u}_{k}-u\opt}$ converges to zero. Thus, we have the
  strong convergence of the sequence~$\ba{\va{U}_{k}}_{k \in \bbN}$ to the
  unique solution $u\opt$ of the problem.
\end{proof}

\subsection{Efficiency estimates}
\label{subsec:eff_est}

In this section, we derive efficiency estimates for the convergence of the
expectation of function values. In Theorem~\ref{thm:cv_rate_expec_avrg}, we
consider the expected function value taken for the averaged iterates following
the technique of
Polyak-Ruppert~\cite{polyak_acceleration_1992,ruppert_efficient_1988}. We take a
step size $\varepsilon_{k}$ of the order $\mathcal{O}\vardelim{k^{-\theta}}$
with $1/2 < \theta < 1$, ensuring the convergence of the algorithm, and leading
to a better convergence rate than with a small step size
$\varepsilon_{k} = \mathcal{O}\vardelim{k^{-1}}$. The efficiency estimate is
obtained using a similar technique as in~\cite{nemirovski_robust_2009} but
without requiring the boundedness of $\Uad$. Moreover, we are able to take into
account the bias on the gradient with the following assumption, inspired
from~\cite{geiersbach_stochastic_2020-1}:
\begin{enumerate}[resume, label=(A\arabic*)]
\item For $k \in \bbN$, let
  $Q_{k} = \mathrm{ess}\sup_{\omega \in \omeg} \norm{\va{R}_{k}(\omega)}$ be the
  essential supremum of $\norm{\va{R}_{k}}$ and assume that:
  \begin{align*}
    \sum_{k \in \bbN} Q_{k}\varepsilon_{k} < \infty \eqfinp
  \end{align*}
  \label{hyp:bound_bias}
\end{enumerate}

We start by a lemma that proves the boundedness of the expectation of the
Lyapunov function. This result will be used multiple times in this part.
\begin{lemma}
  \label{lem:bound_expec_lyap}
  Under Assumptions~\textup{\ref{hyp:sep_banach}-\ref{hyp:mes_w},
    \ref{hyp:jc_lsc}-\ref{hyp:bound_bias}}, the sequence of expectations of the
  Lyapunov function $\ba{\expec{\ell_{u\opt}\bp{\va{u}_{k}}}}_{k \in \bbN}$
  {is bounded and the sequence $\ba{\expec{\J(\va{u}_{k})}}_{k\in \bbN}$
  takes finite values.}
\end{lemma}

\begin{proof}
  {From Inequality~\eqref{eq:stoch-pbsansc-ppa-maj-psik}, using the fact
    that $\norm{\va{R}_{k}}\leq Q_{k}$ almost surely and using the fact that $\J(\va{u}_{k}) - \J(u\opt)$ is nonnegative,
    we obtain
    \begin{equation}
      \label{eq:stoch-pbsansc-ppa-maj-psik-determ-const}
      \bespc{\ell_{u\opt}\bp{\va{u}_{k+1}}}{\tribu{F}_{k}} \leq
      \np{1+\alpha_{k}} \ell_{u\opt}\bp{\va{u}_{k}}
      + \beta_{k}\bespc{\ell_{u\opt}\bp{\va{u}_{k+1}}}{\tribu{F}_{k}}
      + \gamma_{k}
      \eqfinv
    \end{equation}}
  where
  \begin{equation}
    \label{def:coef_det}
    \alpha_{k} = \alpha\varepsilon_{k}^{2} + \frac{2}{b}\varepsilon_{k}Q_{k} \eqsepv
    \beta_{k}  = \beta\varepsilon_{k}^{2} \eqsepv
    \gamma_{k} = (\gamma + \delta Q_{k}^{2})\varepsilon_{k}^{2} + Q_{k}\varepsilon_{k}
    \eqfinp
  \end{equation}
  Then, {taking the extended expectation (all random variables are nonnegative)} on both sides of Inequality~\eqref{eq:stoch-pbsansc-ppa-maj-psik-determ-const} yields:
  \begin{equation*}
    \expec{\ell_{u\opt}\bp{\va{u}_{k+1}}} \leq
    \np{1+\alpha_{k}} \expec{\ell_{u\opt}\bp{\va{u}_{k}}}
    + \beta_{k}\expec{\ell_{u\opt}\bp{\va{u}_{k+1}}}
    + \gamma_{k}\eqfinp
  \end{equation*}
  From~\ref{hyp:sig_seq} and~\ref{hyp:bound_bias}, $\alpha_{k},\ \beta_{k}$ and
  $\gamma_{k}$ are the terms of convergent series. Using a deterministic version
  of Corollary~\ref{cor:robbins-siegmund}, we get that the sequence
  $\ba{\expec{\ell_{u\opt}\bp{\va{u}_{k}}}}_{k \in \bbN}$ converges and is
  bounded.

  {
    Now, we can use again Inequality~\eqref{eq:stoch-pbsansc-ppa-maj-psik}
    and the fact that, from the previous point, $\expec{\ell_{u\opt}\bp{\va{u}_{k+1}}}$ is finite
    to obtain that
    \begin{equation}
      \varepsilon_{k} \expec{\J(\va{u}_{k}) - \J(u\opt)}
      \leq
      \np{1+\alpha_{k}} \expec{\ell_{u\opt}\bp{\va{u}_{k}}}
      + \np{\beta_{k} -1}\expec{\ell_{u\opt}\bp{\va{u}_{k+1}}}
      + \gamma_{k}
      \eqfinv\label{eq:maj_j_uk}
    \end{equation}
    from which we obtain
    {in particular that $\expec{\J(\va{u}_{k}) - \J(u\opt)}$ is finite.}
  }
\end{proof}

\begin{theorem}
  Suppose that Assumptions~\textup{\ref{hyp:sep_banach}-\ref{hyp:mes_w},
    \ref{hyp:jc_lsc}-\ref{hyp:bound_bias}} are satisfied. Let $n \in \bbN$ and
  let $\ba{\va{u}_{k}}_{k\in \bbN}$ be the sequence of iterates of the
  stochastic APP algorithm. Define the averaged iterate as:
  \begin{align*}
    \widetilde{\va{U}}_{i}^{n} = \sum_{k=i}^{n} \eta^{i}_{k} \va{U}_{k}
    \quad
    \text{with}
    \quad
    \eta^{i}_{k} = \frac{\varepsilon_{k}}{\sum_{l=i}^{n}\varepsilon_{l}}
     \eqfinp
  \end{align*}
  Suppose that for all $k \in \bbN, \ \varepsilon_{k} = ck^{-\theta}$ with
  $1/2 < \theta < 1$ and a constant $c > 0$. Then for any minimizer $u\opt$ of
  $J$, we have:
  \begin{align*}
    \expec{\J\left(\widetilde{\va{U}}_{1}^{n}\right) - J\left(u\opt\right)} = \mathcal{O}\vardelim{n^{\theta-1}} \eqfinp
  \end{align*}
  In particular, the rate of convergence can be arbitrarily close to the order
  $n^{-1/2}$ if $\theta$ is chosen to be arbitrarily close to $1/2$.
  \label{thm:cv_rate_expec_avrg}
\end{theorem}

\begin{proof}
  From Lemma~\ref{lem:bound_expec_lyap}, we get that
  Inequality~\eqref{eq:maj_j_uk} is satisfied and the sequence
  $\ba{\expec{\ell_{u\opt}\bp{\va{u}_{k}}}}_{k \in \bbN}$ is bounded. Then,
  there exists a constant $M \geq 0$ such that
  $\expec{\ell_{u\opt}\bp{\va{u}_{k}}} \leq M$ for all $k \in
  \bbN$. Summing~\eqref{eq:maj_j_uk} over $i \leq k \leq n$ and using
  $\expec{\ell_{u\opt}\bp{\va{u}_{k}}} \leq M$, we get:
  \begin{align*}
    \sum_{k=i}^{n} \varepsilon_{k} \jconv{k} &\leq {M + } \sum_{k=i}^{n} \vardelim{M(\alpha + \beta) + \gamma + \delta Q_{k}^{2}}\varepsilon_{k}^{2} + \vardelim{\frac{2}{b}M + 1} Q_{k}\varepsilon_{k}\eqfinp
  \end{align*}
  In the sequel, let $R = M(\alpha + \beta) + \gamma$ and
  $S = \frac{2}{b}M + 1$. By convexity of $J$, we get:
  \begin{align*}
    \expec{\J\vardelim{\widetilde{\va{U}}_{i}^{n}} - \J(u\opt)} \leq \frac{{M + } \sum_{k=i}^{n} \vardelim{R + \delta Q_{k}^{2}}\varepsilon_{k}^{2} + S Q_{k}\varepsilon_{k}}{\sum_{k=i}^{n} \varepsilon_{k}}\eqfinp
  \end{align*}
  We have $\varepsilon_{k} = ck^{-\theta}$ with $1/2 < \theta < 1$ and:
  \begin{align*}
    \sum_{k=1}^{n} k^{-\theta} \geq \frac{(n+1)^{1-\theta}-1}{1-\theta} \geq
    \widetilde{C}_{\theta}n^{1-\theta} \eqfinv
  \end{align*}
  for some $\widetilde{C}_{\theta} > 0$. Moreover, from~\ref{hyp:sig_seq}
  and~\ref{hyp:bound_bias}, $\varepsilon_{k}^{2}$, $Q_{k}\varepsilon_{k}$ and
  $Q_{k}^{2}\varepsilon_{k}^{2}$ are the terms of convergent series. Thus, there
  exists a constant $C_{\theta} > 0$ such that:
  \begin{align*}
    \expec{\J\vardelim{\widetilde{\va{U}}_{1}^{n}} - \J(u\opt)} \leq \frac{C_{\theta}}{n^{1-\theta}}\eqfinv
  \end{align*}
  which gives the desired rate of convergence.
\end{proof}

Theorem~\ref{thm:cv_rate_expec_avrg} proves a convergence rate of order
$\mathcal{O}\vardelim{n^{\theta-1}}$ for the stochastic APP algorithm without
assuming strong convexity of the objective. This rate appears for stochastic
gradient descent in~\cite{bach_non-asymptotic_2011} where it is stated that the
combination of large step sizes of order $\mathcal{O}\vardelim{n^{-\theta}}$
with $1/2 < \theta < 1$, together with averaging lead to the best convergence
behavior. A similar rate is also given for stochastic proximal gradient
in~\cite{rosasco_convergence_2019}.

In the following theorem, we show that this rate also holds when we consider the
expected function value taken at the last iterate $\va{U}_{n}$ instead of the
averaged iterate $\widetilde{\va{U}}_{1}^{n}$. Using the concept of modified
Fej\'{e}r monotone sequences, Lin and al. \cite{lin_modified_2018} have given
convergence rates of the expected function value of the last
iterate for many algorithms, such as the projected subgradient method or the
proximal gradient algorithm. The idea of modified Fej\'{e}r sequence is adapted
to the stochastic case in~\cite[Theorem 3.1]{rosasco_convergence_2019}. We
further adapt this concept for the stochastic APP algorithm. {Before
stating the result, we need a technical lemma.
\begin{lemma}
  Under Assumptions~\textup{\ref{hyp:sep_banach}-\ref{hyp:mes_w},
  \ref{hyp:jc_lsc}-\ref{hyp:bound_bias}}, there exists a constant $M' \geq 0$ such that $\expec{\ell_{\va{u}_{j}}\bp{\va{u}_{k}}} \leq M'$ for all $j,k \in \bbN$.
\label{lem:unif_bd_expec_lyap}
\end{lemma}
}

{
\begin{proof}
  As $\nabla K$ is $L_{K}$-Lipschitz continuous by~\ref{hyp:aux_func},
  we have the following inequality
  (see for example~\cite[Lemma 1.2.3]{nesterov_introductory_2004})
  \begin{equation*}
  K(v) \leq K(u) + \bscal{\nabla K(u)}{v-u} + \frac{L_{K}}{2} \sqnorm{u-v} \eqfinv
  \end{equation*}
  and hence, for all $u,v \in \Uad$,
  \begin{equation*}
    \ell_{v}\np{u} \leq \frac{L_{K}}{2}\nnorm{u-v}^{2} \leq
    L_{K} \bp{\nnorm{u-u\opt}^{2} + \nnorm{v-u\opt}^{2}} \eqfinv
  \end{equation*}
  the last inequality arising from the standard norm inequality
  $\nnorm{a+b}^{2} \leq 2\np{\nnorm{a}^{2}+\nnorm{b}^{2}}$. Then using the~$b$-strong
  convexity of~$K$ and Equation~\eqref{eq:stoch-pbsansc-ppa-lyap}, we obtain
  \begin{equation}
  \ell_{v}\np{u} \leq \frac{2 L_{K}}{b} \bp{\ell_{u\opt}\np{u}+\ell_{u\opt}\np{v}} \eqfinp
  \label{eq:up_bd_lyap}
  \end{equation}
  Writing Inequality~\eqref{eq:up_bd_lyap} in terms of random variables, $v$
  and $u$ being respectively replaced by $\va{U}_{j}$ and $\va{U}_{k}$, and
  taking the extended expectation (all quantities are positive) leads to:
  \begin{align*}
    \expec{\ell_{\va{u}_{j}}\np{\va{u}_{k}}}
    & \leq \frac{2 L_{K}}{b} \Bp{\expec{\ell_{u\opt}\np{\va{u}_{k}}}
      + \expec{\ell_{u\opt}\np{\va{u}_{j}}}} \eqfinv \\
    & \leq \frac{4 L_{K}}{b} M \eqfinv
  \end{align*}
  since we have by Lemma~\ref{lem:bound_expec_lyap} that the quantities
  $\expec{\ell_{u\opt}\np{\va{u}_{k}}}$ are bounded by a constant~$M$.
\end{proof}
}

\begin{theorem}
  Suppose that Assumptions~\textup{\ref{hyp:sep_banach}-\ref{hyp:mes_w},
    \ref{hyp:jc_lsc}-\ref{hyp:bound_bias}} are satisfied. Let $n \in \bbN$ and
  let $\ba{\va{u}_{k}}_{k\in \bbN}$ be the sequence of iterates of the
  stochastic APP algorithm. Suppose that for all $k \in \bbN$,
  $\varepsilon_{k} = ck^{-\theta}$ with $1/2 < \theta < 1$ and a constant
  $c > 0$. Assume also that $Q_{k} \leq qk^{-\nu}$ for $\nu > 1 - \theta$ and a
  constant $q > 0$. Then, for any minimizer $u\opt$ of $J$ we have:
  \begin{align*}
    \jconv{n} = \mathcal{O}\vardelim{n^{\theta-1}} \eqfinp
  \end{align*}
  In particular, the rate of convergence can be arbitrarily close to the order
  $n^{-1/2}$ if $\theta$ is chosen to be arbitrarily close to $1/2$.
  \label{thm:cv_rate_expec_last}
\end{theorem}

\begin{proof} For $k\in \bbN$,
  let $a_{k} = \jconv{k}$, {which is {finite}
  by Lemma~\ref{lem:bound_expec_lyap}}.
  Fix $n \in \bbN$, from Lemma~\ref{lem:sum_manip} applied to the sequence
  $\na{\varepsilon_{k} a_{k}}_{k\in \bbN}$ (see Appendix), we have that:
  \begin{align*}
    \varepsilon_{n} a_{n} = \frac{1}{n} \bp{ \sum_{k=1}^{n} \varepsilon_{k} a_{k}}
    + \sum_{i=1}^{n-1} \frac{1}{i(i+1)} \bgp{ \Bp{\sum_{k=n-i+1}^{n} \varepsilon_{k} a_{k}} - i\varepsilon_{n-i}a_{n-i}}\eqfinp
  \end{align*}
  Moreover, we have that:
  \begin{align*}
    \Bp{\sum_{k=n-i+1}^{n} \varepsilon_{k} a_{k}} - i\varepsilon_{n-i}a_{n-i}
    &=
      \sum_{k=n-i+1}^{n} \np{ \varepsilon_{k} a_{k} - \varepsilon_{n-i}a_{n-i}}
    \\
    &=
      \sum_{k=n-i+1}^{n} \bp{ \varepsilon_{k}\np{ a_{k} -a_{n-i}} + \np{\varepsilon_{k} - \varepsilon_{n-i}} a_{n-i}}\\
    &\le
      \sum_{k=n-i+1}^{n} \varepsilon_{k}\np{ a_{k} -a_{n-i}}\\
      &= \sum_{k=n-i+1}^{n} \varepsilon_{k}\jconvany{k}{n-i}
      \eqfinv
  \end{align*}
  where the last inequality follows from the fact that the sequence $\na{\varepsilon_{k}}_{k \in \bbN}$ is decreasing and that by optimality of $u\opt$, we have $a_{n-i} \geq 0$. We therefore obtain that:
  \begin{align}
    \varepsilon_{n} a_{n} \leq \frac{1}{n} \Bp{\sum_{k=1}^{n} \varepsilon_{k} a_{k}}
    + \sum_{i=1}^{n-1} \frac{1}{i(i+1)} \sum_{k=n-i+1}^{n} \varepsilon_{k}{\jconvany{k}{n-i}} \eqfinp
    \label{eq:init_bd}
  \end{align}
  {We bound the terms of~\eqref{eq:init_bd}:}
  \begin{enumerate}
    \item  From Lemma~\ref{lem:bound_expec_lyap}, Inequality~\eqref{eq:maj_j_uk} is
  satisfied and there exists a constant $M \geq 0$ such that
  $\expec{\ell_{u\opt}\bp{\va{u}_{k}}} \leq M$ for all $k \in \bbN$.
  Summing~\eqref{eq:maj_j_uk} from $1$ to $n$ and
    using that $\beta_{n} - 1 \leq 0$ for $n$ large enough, we get:
  \begin{align}
  \begin{aligned}
    \sum_{k=1}^{n} \varepsilon_{k} a_{k}&
    {\begin{multlined}[t]
    \leq M_{1} + \sum_{k=2}^{n} \Bp{(\alpha_{k} + \beta_{k-1})\expec{\ell_{u\opt}\bp{\va{u}_{k}}} + \gamma_{k}} \\
    + (\beta_{n} - 1)\expec{\ell_{u\opt}\bp{\va{u}_{n+1}}}
    \end{multlined}}\\
    &\leq M_{1} + \sum_{k=2}^{n} M(\alpha_{k} + \beta_{k-1}) + \gamma_{k}\eqfinp
    \label{eq:term_bd}
  \end{aligned}
  \end{align}
  with $M_{1} = (1+\alpha_{1})\ell_{u\opt}\vardelim{u_{1}} + \gamma_{1}$.

\item
  In order to bound the second term in the right hand side of~\eqref{eq:init_bd},
  we start from Inequality~\eqref{eq:bound_real} in Lemma~\ref{lem:lyap_bd}
  that we write in terms of random variables, $v$ being replaced by~$\va{u}_{n{-}i}$.
  This leads to the following inequality
  \begin{multline}
    \label{eq:inequality-ps1}
    \underbrace{ (1-\beta\varepsilon_{k}^{2}) \ell_{\va{u}_{n{-}i}}(\va{u}_{k+1})}_{\va{A}'_{k+1}} \leq
    \underbrace{\Bp{1+\alpha\varepsilon_{k}^{2}
        + \frac{2}{b}\varepsilon_{k}\nnorm{\va{r}_{k}}} \ell_{\va{u}_{n{-}i}}(\va{u}_{k})}_{\va{B}'_k}
    \\
    + \underbrace{\Bp{\gamma\varepsilon_{k}^{2}+\varepsilon_{k}\nnorm{\va{r}_{k}}
        +\delta\np{\varepsilon_{k}\nnorm{\va{r}_{k}}}^{2}}}_{\va{C}'_k} \\
    - \underbrace{\varepsilon_{k} \bp{(\jc+\jad)(\va{u}_{k},\va{w}_{k+1})-(\jc+\jad)(\va{u}_{n{-}i},\va{w}_{k+1})}}_{\va{D}'_k}
    \eqfinp
  \end{multline}
  This inequality is similar to Inequality~\eqref{eq:inequality-ps},
  but the term ${\va{D}'_k}$ cannot be assumed to be  nonnegative
  and we need to adapt the steps used in the proof of Theorem~\ref{thm:cv}.
  Using Lemma~\ref{lem:unif_bd_expec_lyap} we obtain that the terms
  ${\va{A}'_{k+1}}$ and ${\va{B}'_k}$ are integrable so that
  their conditional expectation is well defined.
  Moreover, since~$k \geq n{-}i$, the random variable~$\va{u}_{n-i}$
  is $\tribu{F}_k$-measurable and thus~$\ell_{\va{u}_{n{-}i}}\bp{\va{u}_{k}}$
  is also $\tribu{F}_k$-measurable. For the term ${\va{D}'_k}$, we consider
  separately ${\va{D}^{\prime,1}_k}=\varepsilon_{k} {(\jc+\jad)(\va{u}_{k},\va{w}_{k+1})}$
  and ${\va{D}^{\prime,2}_k}=\varepsilon_{k} {(\jc+\jad)(\va{u}_{n-i},\va{w}_{k+1})}$.
  Following the steps of the proof of Theorem~\ref{thm:cv}, we obtain
  that both terms ${\va{D}^{\prime,1}_k}$ and ${\va{D}^{\prime,2}_k}$
  admit explicit (extended) conditional expectations with respect to
  the $\sigma$-field $\tribu{F}_{k}$ which are given by
  $\bespc{ \va{D}^{\prime,1}_k}{\tribu{F}_{k}}(\omega)
   = \varepsilon_{k}\J(\va{u}_k(\omega))$
  and $\bespc{ \va{D}^{\prime,2}_k}{\tribu{F}_{k}}(\omega)
   = \varepsilon_{k}\J(\va{u}_{n-i}(\omega))$.
  Now, using Lemma~\ref{lem:bound_expec_lyap} we have that $J(\va{u}_k)$
  and $J(\va{u}_{n-i})$ are both integrable and therefore the conditional
  expectation of ${\va{D}'_k}$ is well defined as the difference of two
  integrable functions. We can therefore take the conditional expectation
  \wrt~$\tribu{F}_k$ on both sides of Inequality~\eqref{eq:inequality-ps1}
  and we obtain
  \begin{multline*}
    \bespc{\ell_{\va{u}_{n{-}i}}\bp{\va{u}_{k+1}}}{\tribu{F}_{k}}
    \leq
    \np{1+\va{\alpha}_{k}} \ell_{\va{u}_{n{-}i}}\bp{\va{u}_{k}}
    + \va{\beta}_{k}\bespc{\ell_{\va{u}_{n{-}i}}\bp{\va{u}_{k+1}}}{\tribu{F}_{k}} \\
    + \va{\gamma}_{k}
    - \varepsilon_{k} \bp{\J(\va{u}_{k}) - \J(\va{u}_{n{-}i})}
    \eqfinp
  \end{multline*}
  In this last inequality, replacing the random variables
  $\np{\va{\alpha}_{k},\va{\beta}_{k}, \va{\gamma}_{k}}$ by their
  deterministic upper bounds $\np{{\alpha}_{k},{\beta}_{k},{\gamma}_{k}}$
  given by Equation~\eqref{def:coef_det}, and taking the expectation
  of the left and right terms, we obtain:
  \begin{multline}
    \label{eq:recur_esp_lyap_ni_unordered}
    \expec{\ell_{\va{u}_{n-i}}\bp{\va{u}_{k+1}}} \leq
    \np{1+\alpha_{k}} \expec{\ell_{\va{u}_{n-i}}\bp{\va{u}_{k}}}
    + \beta_{k}\expec{\ell_{\va{u}_{n-i}}\bp{\va{u}_{k+1}}} \\
    + \gamma_{k}
    - \varepsilon_{k} \jconvany{k}{n-i} \eqfinv
  \end{multline}
  We have already seen by Lemma~\ref{lem:unif_bd_expec_lyap} that
  $\expec{\ell_{\va{u}_{n-i}}\bp{\va{u}_{k+1}}}$ is finite. We have
  also seen by Lemma~\ref{lem:bound_expec_lyap} that $a_{k} = \jconv{k}$
  is finite, therefore so is $\jconvany{k}{n-i} = a_{k} - a_{n-i}$.
  All quantities in~\eqref{eq:recur_esp_lyap_ni_unordered}
  are finite and we can write:
    \begin{multline}
    \label{eq:recur_esp_lyap_ni}
    \varepsilon_{k} \jconvany{k}{n-i} \leq
    \np{1+\alpha_{k}} \expec{\ell_{\va{u}_{n-i}}\bp{\va{u}_{k}}}\\
    + (\beta_{k} - 1)\expec{\ell_{\va{u}_{n-i}}\bp{\va{u}_{k+1}}}
    + \gamma_{k}
     \eqfinv
  \end{multline}
  From Lemma~\ref{lem:unif_bd_expec_lyap}, there exists a constant
  $M' \geq 0$ such that $\expec{\ell_{\va{u}_{j}}\bp{\va{u}_{k}}} \leq M'$
  for all $j,k \in \bbN$. Summing~\eqref{eq:recur_esp_lyap_ni} from $n-i$
  to $n$ and using that $\beta_{n} - 1 \leq 0$ for $n$ large enough, we get:


  \begin{align}
    \sum_{k=n-i+1}^{n} \varepsilon_{k}{\jconvany{k}{n-i}}
    &\leq \gamma_{n-i} + \sum_{k=n-i+1}^{n} M'\vardelim{\alpha_{k} + \beta_{k-1}} + \gamma_{k}
    \eqfinp
    \label{eq:bd_seond_term}
  \end{align}

\end{enumerate}
  Define:
  \begin{align*}
    \bar{\alpha}_{k} = \alpha\varepsilon_{k}^{2} + \frac{2}{b}\varepsilon_{k}qk^{-\nu}, \quad
    \bar{\gamma}_{k} = \vardelim{\gamma + \delta q^{2}k^{-2\nu}}\varepsilon_{k}^{2} + \varepsilon_{k}qk^{-\nu} \eqfinp
  \end{align*}
  As $Q_{k} \leq qk^{-\nu}$, we have $\alpha_{k} \leq \bar{\alpha}_{k}$ and $\gamma_{k} \leq \bar{\gamma}_{k}$. Let $\xi_{k} = {\vardelim{M+M'}}(\bar{\alpha}_{k} + {\beta_{k-1}}) + \bar{\gamma}_{k}$.
  Moreover, note that:
  \begin{align}
    \sum_{i=1}^{n-1} \frac{1}{i(i+1)} \sum_{k=n-i+1}^{n} \xi_{k}
    = \sum_{k=2}^{n} \sum_{i=n-k+1}^{n-1} \vardelim{\frac{1}{i} - \frac{1}{i+1}} \xi_{k}
    = \sum_{k=2}^{n} \frac{\xi_{k}}{n-k+1} - \frac{1}{n} \sum_{k=2}^{n} \xi_{k}\eqfinp
    \label{eq:exch_sum}
  \end{align}
  We plug~\eqref{eq:term_bd} {and~\eqref{eq:bd_seond_term}} into~\eqref{eq:init_bd} and use that $M(\alpha_{k} + \beta_{k}) + \gamma_{k} \leq \xi_{k}$ {and $M'(\alpha_{k} + \beta_{k}) + \gamma_{k} \leq \xi_{k}$} along with~\eqref{eq:exch_sum}. This yields:
  \begin{align}
    \varepsilon_{n}a_{n} &\leq {\frac{M_{1}}{n}} + \frac{1}{n} \sum_{k=2}^{n} \xi_{k} + \sum_{k=2}^{n} \frac{1}{n-k+1} \xi_{k} - \frac{1}{n} \sum_{k=2}^{n} \xi_{k} + {\sum_{i=1}^{n-1} \frac{\gamma_{n-i}}{i(i+1)}} \notag\\
    &= {\frac{M_{1}}{n}} + \sum_{k=2}^{n} \frac{1}{n-k+1} \xi_{k} +  {\sum_{i=1}^{n-1} \frac{\gamma_{n-i}}{i(i+1)}}\eqfinp \label{eq:bd_en_an}
  \end{align}
  From the assumptions on $\varepsilon_{k}$, $\na{\xi_{k}}_{k \in \bbN}$ is non-increasing. Thus,
  \begin{align}
    \sum_{k=2}^{n} \frac{1}{n-k+1} \xi_{k}
    &\leq \xi_{\left\lfloor{\frac{n}{2}+1}\right\rfloor} \sum_{n/2+1\leq k \leq n} \frac{1}{n-k+1} + \frac{2}{n} \sum_{2\leq k < n/2+1} \xi_{k} \eqfinv\notag\\
    &\leq \xi_{\left\lfloor{\frac{n}{2}+1}\right\rfloor} \vardelim{\log\vardelim{\frac{n}{2}} + 1} + \frac{2}{n} \sum_{k=2}^{n} \xi_{k} \eqfinp \label{eq:sum_manip}
  \end{align}
  By assumption, $\varepsilon_{n} = cn^{-\theta}$ and we have:
  \begin{align*}
    \xi_{k}
    &\leq \vardelim{{(M+M')}(\alpha + \beta) + \gamma + \delta q^{2}k^{-2\nu}}c^{2}{{(k-1)}}^{-2\theta} + \vardelim{\frac{2}{b}{(M+M')} + 1} cqk^{-(\nu+\theta)}\\
    &\leq \xi (k-1)^{-\mu}\eqfinv
  \end{align*}
  for $\mu = \min\left\{2\theta, \nu + \theta\right\}$ and some constant $\xi > 0$ so that,
  \begin{align}
    \frac{1}{\varepsilon_{n}}\sum_{k=2}^{n} \frac{1}{n-k+1} \xi_{k} \leq 2^{\mu}\frac{\xi}{c} {(n-1)}^{\theta - \mu} \vardelim{\log\vardelim{\frac{n}{2}} + 1} + 2\frac{\xi}{c} n^{\theta - 1} \sum_{k=1}^{n-1} k^{-\mu}\eqfinp
    \label{eq:bd_xik}
  \end{align}
  As $\theta > 1/2$ and $\nu > 1 - \theta$, we have $\mu > 1$ so,
  \begin{align*}
    \sum_{k=1}^{n-1} k^{-\mu} \leq \frac{\mu}{\mu - 1}\eqfinp
  \end{align*}
  {
  Using a similar computation as in~\eqref{eq:sum_manip} and that $\gamma_{i} \leq \kappa i^{-\mu}$ for some constant $\kappa > 0$, we deduce that there exist constants $\Gamma_{1}, \Gamma_{2} > 0$ such that:
  \begin{align}
    \frac{1}{\varepsilon_{n}}\sum_{i=1}^{n-1} \frac{\gamma_{n-i}}{i(i+1)} \leq \Gamma_{1} n^{\theta - \mu} + \Gamma_{2} n^{\theta - 2} \eqfinp
    \label{eq:bd_gamma}
  \end{align}
  Gathering~\eqref{eq:bd_xik} and~\eqref{eq:bd_gamma} into~\eqref{eq:bd_en_an}, we get:
  \begin{align}
    a_{n} \leq M_{1}n^{\theta - 1} + \Xi_{1} (n-1)^{\theta - \mu} \vardelim{\log\vardelim{\frac{n}{2}} + 1} + \Xi_{2} n^{\theta - 1} + \Gamma_{1} n^{\theta - \mu} + \Gamma_{2} n^{\theta - 2} \eqfinv
  \end{align}
  with $\Xi_{1} = 2^{\mu}\frac{\xi}{c}$ and $\Xi_{2} = 2\frac{\xi}{c} \frac{\mu}{\mu-1}$.
  }
  Finally, as $\theta - \mu < \theta - 1$, we get that $a_{n} = \mathcal{O}\vardelim{n^{\theta - 1}}$.
  This concludes the proof.
\end{proof}

\begin{remark}
  Inequality~\eqref{eq:maj_j_uk} (which holds in fact for any $u\in \Uad$
  in place of $u\opt$) is the counterpart of modified Fej\'{e}r
  monotonicity~\cite{lin_modified_2018}. The main differences are
  that~\eqref{eq:maj_j_uk} involves a Bregman divergence instead of the
  Euclidean distance. Moreover, there are coefficients
  $\alpha_{k}, \beta_{k} > 0$ that slightly degrade the inequality compared to
  what we obtain with Fej\'{e}r monotone sequences where
  $\alpha_{k} = \beta_{k} = 0$. The summability of $\alpha_{k}$ and $\beta_{k}$
  in addition with the boundedness of the expectation of the Bregman divergence
  $\ba{\expec{\ell_{u\opt}\bp{\va{u}_{k}}}}_{k \in \bbN}$ allow us to proceed in
  the same way as in~\cite{lin_modified_2018,rosasco_convergence_2019} to get
  the convergence rate of Theorem~\ref{thm:cv_rate_expec_last}.
\end{remark}

\section{Conclusion}
\label{sec:ccl}

We have studied the stochastic APP algorithm in a reflexive {separable} Banach
case. This framework generalizes many stochastic optimization algorithms
{for convex problems}. We have proved the measurability of the iterates
of the algorithm, hence filling a theoretical gap to ensure that the quantities
we manipulate when deriving efficiency estimates are well-defined. We have shown
the convergence of the stochastic APP algorithm in the case where a bias on the
gradient is considered. Finally, efficiency estimates are derived while taking
the bias into account. Assuming a sufficiently fast decay of this bias, we get a
convergence rate for the expectation of the function values that is similar to
that of well-known stochastic optimization algorithms when no bias is present,
such as stochastic gradient descent~\cite{bach_non-asymptotic_2011}, stochastic
mirror descent~\cite{nemirovski_robust_2009} or the stochastic proximal gradient
algorithm~\cite{rosasco_convergence_2019}. Future work will consist in an
application the stochastic APP algorithm to an optimization problem in a Banach
space with decomposition aspects in mind.

\section*{Acknowledgements}
The authors would like to thank the two anonymous reviewers for their
careful reading and for their comments and questions which helped
to improve and clarify the manuscript.

\appendix

\section{Technical results used in the proofs}

\begin{lemma}
  \label{lem:sum_manip}
  Let $\na{a_{i}}_{i \in \bbN}$ be a sequence in $\bbR$. Let $n \in \bbN$ and for $i \in \na{0, 1, \ldots, n-1}$, let $s_{i} = \sum_{k=n-i}^{n} a_{k}$. Then,
  \begin{align*}
    a_{n} = \frac{s_{n-1}}{n} + \sum_{i=1}^{n-1} \frac{1}{i(i+1)}(s_{i-1} - ia_{n-i}) \eqfinp
  \end{align*}
\end{lemma}
The proof of Lemma~\ref{lem:sum_manip} is a straightforward computation and is left to the reader.

\begin{theorem}
  \label{thm:robbins-siegmund}
  \cite[Robbins-Siegmund]{robbins_convergence_1971} Consider four sequences of nonnegative random variables $\{\va{\Lambda}_{k}\}_{k\in\mathbb{N}}$,
  $\{\va{\alpha}_{k}\}_{k\in\mathbb{N}}$,
  $\{\va{\beta}_{k}\}_{k\in\mathbb{N}}$
  and~$\{\va{\eta}_{k}\}_{k\in\mathbb{N}}$,
  that are all adapted to a given filtration
  $\{\tribu{F}_{k}\}_{k\in\mathbb{N}}$.
  Moreover, suppose that:
  \begin{align*}
    \bespc{\va{\Lambda}_{k+1}}{\tribu{F}_{k}} \leq
    (1+\va{\alpha}_{k})\va{\Lambda}_{k} + \va{\beta}_{k} -
    \va{\eta}_{k} , \; \forall k \in \mathbb{N} \; ,
  \end{align*}
  and that $\sum_{k} \va{\alpha}_{k} < + \infty$ and $\sum_{k} \va{\beta}_{k} < + \infty$ almost surely.
  Then, the sequence of random variables~$\{\va{\Lambda}_{k}\}_{k\in\mathbb{N}}$ converges almost surely to a finite random variable $\va{\Lambda}^{\infty}$,
  and we have in addition that $\sum_{k} \va{\eta}_{k} < + \infty$ almost surely.
\end{theorem}
An extension of Robbins-Siegmund theorem is given by the following corollary.
\begin{corollary}
  \label{cor:robbins-siegmund}
  Consider the following sequences of nonnegative random variables $\{\va{\Lambda}_{k}\}_{k\in\mathbb{N}}$,
  $\{\va{\alpha}_{k}\}_{k\in\mathbb{N}}$,
  $\{\va{\beta}_{k}\}_{k\in\mathbb{N}}$,
  $\{\va{\gamma}_{k}\}_{k\in\mathbb{N}}$,
  and~$\{\va{\eta}_{k}\}_{k\in\mathbb{N}}$,
  that are all adapted to a given filtration
  $\{\tribu{F}_{k}\}_{k\in\mathbb{N}}$.
  Moreover suppose that:
  \begin{align*}
    \bespc{\va{\Lambda}_{k+1}}{\tribu{F}_{k}} \leq
    \bp{1+\va{\alpha}_{k}}\va{\Lambda}_{k}
    + \va{\beta}_{k}\bespc{\va{\Lambda}_{k+1}}{\tribu{F}_{k}}
    + \va{\gamma}_{k} - \va{\eta}_{k} \; ,
  \end{align*}
  and that $\sum_{k} \va{\alpha}_{k} < + \infty$, $\sum_{k} \va{\beta}_{k} < + \infty$ and $\sum_{k} \va{\gamma}_{k} < + \infty$ almost surely.
  Then, the sequence of random variables $\{\va{\Lambda}_{k}\}_{k\in\mathbb{N}}$ converges almost surely to a finite random variable $\va{\Lambda}^{\infty}$, and we have in addition that $\sum_{k} \va{\eta}_{k} < + \infty$ almost surely.
\end{corollary}

\begin{proof}
  Consider a realization of the different sequences satisfying the assumptions of the corollary, and define three sequences
  $\{\widetilde{\alpha}_{k}\}_{k\in\mathbb{N}}$,
  $\{\widetilde{\gamma}_{k}\}_{k\in\mathbb{N}}$
  and~$\{\widetilde{\eta}_{k}\}_{k\in\mathbb{N}}$ such that:
  \begin{align*}
    1+\widetilde{\alpha}_{k} = \frac{1+\alpha_{k}}{1-\beta_{k}} \; , \enspace
    \widetilde{\gamma}_{k} = \frac{\gamma_{k}}{1-\beta_{k}} \; , \enspace
    \widetilde{\eta}_{k} = \frac{\eta_{k}}{1-\beta_{k}} \; .
  \end{align*}
  As the sequence~$\na{\beta_{k}}$ converges to zero,
  we have that $\beta_{k} \leq 1/2$ for $k$ large enough.
  For such $k$, we get:
  \begin{align*}
    \frac{1}{1-\beta_{k}} \leq 1 + 2 \beta_{k}
   \quad \text{and} \quad
    1 \leq \frac{1}{1-\beta_{k}} \leq 2 \; .
  \end{align*}
  Then, we deduce that
  $\widetilde{\alpha}_{k} \leq 2(\alpha_{k}+\beta_{k})$,
  $\widetilde{\gamma}_{k} \leq 2 \gamma_{k}$ and
  $\widetilde{\eta}_{k} \geq \eta_{k}$.
  The conclusions of the corollary are then obtained by applying Theorem~\ref{thm:robbins-siegmund} directly.
\end{proof}

\begin{proposition}
  \label{prop:sglb-lip}
  {Let $\ban$ be a Banach space and} consider a function $J : \ban \rightarrow \mathbb{R}$ that is subdifferentiable on a non-empty, closed, convex subset $\Uad$ of $\ban$, with linearly bounded subgradient. Then, there exist $c_{1} > 0$ and $c_{2} > 0$ such that:
  \begin{align}
  \label{eq:sglb-lip}
    \forall (u,v) \in \Uad\times\Uad \eqsepv
    \babs{\J(u)-\J(v)} \leq \Bp{c_{1}\max\ba{\nnorm{u},\nnorm{v}}+c_{2}} \: \nnorm{u-v} \eqfinp
  \end{align}
  In particular, $J$ is Lipschitz continuous on every bounded subset that is contained in~$\Uad$.
\end{proposition}

\begin{proof}
  Let~$(u,v) \in \Uad\times\Uad$.
  From the definition of subdifferentiability, we get that for all $r \in \partial\J(u)$ and for all $s \in \partial\J(v)$:
  \begin{align*}
    \nscal{s}{u-v} \leq\J(u) -\J(v) \leq \nscal{r}{u-v} \eqfinv
  \end{align*}
  and therefore:
  \begin{align*}
    \babs{\J(u) -\J(v)} \leq \max \ba{\nscal{r}{u-v},\nscal{s}{v-u}} \eqfinp
  \end{align*}
  Using Cauchy-Schwarz inequality and the linearly bounded subgradient assumption,
  we get the desired result.
\end{proof}

\begin{proposition}
  \label{prop:tech2}
  {Let $\ban$ be a Banach space and} $J: \ban \rightarrow \bbR$ be a Lipschitz continuous function with constant~$L > 0$.
  Let~$\{u_{k}\}_{k\in\bbN}$ be a sequence of elements in~$\ban$ and let~$\{\varepsilon_{k}\}_{k\in\bbN}$ be
  a real positive sequence such that:
  \begin{enumerate}[label=(\alph*)]
    \item \label{hyp:a} $\sum_{k\in\bbN} \varepsilon_{k} = + \infty$,
    \item \label{hyp:b} $\exists \mu \in \bbR , \;
    \sum_{k\in \bbN} \varepsilon_{k} \abs{\J(u_{k})-\mu} < + \infty$,
    \item \label{hyp:c} $\exists \delta > 0 , \; \forall k \in \bbN , \;
    \nnorm{u_{k+1}-u_{k}} \leq \delta \varepsilon_{k}$.
  \end{enumerate}
  Then, the sequence~$\ba{\J(u_{k})}_{k\in\bbN}$ converges to~$\mu$.
\end{proposition}

\begin{proof}
  For $\alpha > 0$, define $N_{\alpha} = \ba{ k \in \bbN , \ \abs{\J(u_{k})-\mu} \leq \alpha }$ and $N_{\alpha}^{\complement} = \bbN \setminus N_{\alpha}$.
  \begin{itemize}
    \item[(i)] From Assumption~\ref{hyp:b}, we have:
    \begin{align*}
      + \infty > \sum_{k\in \bbN} \varepsilon_{k} \babs{\J(u_{k})-\mu} \geq
      \sum_{k \in N_{\alpha}^{\complement}} \varepsilon_{k} \babs{\J(u_{k})-\mu} \geq
      \alpha \sum_{k \in N_{\alpha}^{\complement}} \varepsilon_{k} \; .
    \end{align*}
    Hence, for all $\beta > 0$, there exists $n_{\beta} \in \bbN$ such that $\sum_{k \geq n_{\beta} , k \in N_{\alpha}^{\complement}} \varepsilon_{k} \leq \beta$.
    \item[(ii)] From Assumption~\ref{hyp:a}, we have $\sum_{k\in \bbN} \varepsilon_{k}=\sum_{k \in N_{\alpha}} \varepsilon_{k}+\sum_{k \in N_{\alpha}^{\complement}} \varepsilon_{k} = +\infty$
    but we have just proved that the last sum in the above equality is finite, hence the first sum of the right hand side is infinite, which implies that~$N_{\alpha}$ is infinite.
  \end{itemize}
  Let~$\epsilon > 0$, choose~$\alpha = \epsilon/2$ and
  $\beta = \epsilon/(2 L \delta)$. Let~$n_{\beta}$ be the integer defined in~(i).
  For~$k \geq n_{\beta}$, there are two possible cases:
  \begin{itemize}
    \item $k \in N_{\alpha}$: then, by definition of $N_{\alpha}$, we have $\babs{\J(u_{k})-\mu} \leq \alpha < \epsilon$.
    \item $k \notin N_{\alpha}$: let~$m$ be the smallest element of~$N_{\alpha}$ such that~$m \geq k$, this element exists by~(ii).
    Using the fact that~$J$ is Lipschitz continuous, we get:
    \begin{align*}
      \babs{\J(u_{k})-\mu}
      & \leq \babs{\J(u_{k})-\J(u_{m})} + \babs{\J(u_{m})-\mu} \leq L \nnorm{u_{k}-u_{m}} + \alpha \eqfinp
    \end{align*}
    Now, with Assumption~\ref{hyp:c} and condition~(i), it comes:
    \begin{align*}
      \babs{\J(u_{k})-\mu} \leq L \delta \Bigg( \sum_{l=k}^{m-1} \varepsilon_{l} \Bigg) + \alpha \leq L \delta \Bigg( \sum_{l \geq n_{\beta} , l \in N_{\alpha}^{\complement}}  \varepsilon_{l} \Bigg) + \alpha \leq \epsilon \; .
    \end{align*}
  \end{itemize}
  Hence, we get $\babs{\J(u_{k})-\mu} \leq \epsilon$ for all $k\geq n_{\beta}$, giving the desired result.
\end{proof}

\end{document}

%% file: commandes.tex
\def\mathscr{\EuScript}

\newcommand{\defegal}{:=}                                   

\newcommand{\bbN}{\mathbb{N}}                               
\newcommand{\bbR}{\mathbb{R}}                               

\newcommand{\abs}[1]{\left|#1\right|}                       
\newcommand{\norm}[1]{\left\|#1\right\|}                    
\newcommand{\sqnorm}[1]{\left\|#1\right\|^{2}}              

\newcommand{\projop}[1]{\mathrm{proj}_{#1}}                 
\newcommand{\proj}[2]{\projop{#1}\left(#2\right)}           

\newcommand{\dual}{^{\star}}                                
\newcommand{\opt}{^{\sharp}}                                
\newcommand{\ad}{^{\mathrm{ad}}}                            

\newcommand{\np}[1]{(#1)}                                   
\newcommand{\bp}[1]{\big(#1\big)}                           
\newcommand{\Bp}[1]{\Big(#1\Big)}                           
\newcommand{\bgp}[1]{\bigg(#1\bigg)}                        

\newcommand{\bc}[1]{\big[#1\big]}                           

\newcommand{\na}[1]{\{#1\}}                                 
\newcommand{\ba}[1]{\big\{#1\big\}}                         

\newcommand{\partie}[1]{#1}                                 
\newcommand{\argmin}{\mathop{\arg\min}}                     
\newcommand{\gradi}[2][]{\nabla_{#1}#2}                     

\newcommand{\espacea}[1]{\mathbb{#1}}                       
\newcommand{\tribu}[1]{\mathscr{#1}}                        
\newcommand{\omeg}{\Omega}                                  
\newcommand{\trib}{\tribu{A}}                               
\newcommand{\prbt}{\mathbb{P}}                              
\newcommand{\espe}{\mathbb{E}}                              

\makeatletter
\def\va@a{\boldsymbol{\va@arg^{\textstyle\text{\unboldmath$\scriptstyle\va@expo$}}_{\textstyle\text{\unboldmath$\scriptstyle\va@index$}}}}
\def\va#1{\def\va@expo{}\def\va@index{}\def\va@arg{\uppercase{#1}}%
  \@ifnextchar^{\va@h}{\@ifnextchar_\va@u\va@a}}
\def\va@h^#1{\def\va@expo{#1}\@ifnextchar_\va@hu\va@a}
\def\va@u_#1{\def\va@index{#1}\@ifnextchar^\va@uh\va@a}
\def\va@hu_#1{\def\va@index{#1}\va@a}
\def\va@uh^#1{\def\va@expo{#1}\va@a}
\makeatother

\newcommand{\bigdelim}[1]{\bp{#1}}                          
\newcommand{\vardelim}[1]{\left(#1\right)}                  

\newcommand{\bigdelims}[2]{\bigdelim{#1\ \big|\ #2}}        %
\newcommand{\besp}[2][]{\espe_{#1}\bigdelim{#2}}            

\newcommand{\bespc}[3][]{\espe_{#1}\bigdelims{#2}{#3}}      

\newcommand{\proscal}[2]{\left\langle#1\:,#2\right\rangle}  

\newcommand{\nscal}[2]{\langle#1\:,#2\rangle}               
\newcommand{\bscal}[2]{\big\langle#1\:,#2\big\rangle}       

\newcommand{\babs}[1]{\big|#1\big|}                         

\newcommand{\nnorm}[1]{\|#1\|}                              
\newcommand{\bnorm}[1]{\big\|#1\big\|}                      

\newcommand{\epi}{\mathop{\mathrm{epi}}}                    

\newcommand{\lsc}{\text{l.s.c.}}                            
\newcommand{\wrt}{\text{w.r.t.}}                            
\newcommand{\as}{\text{a.s.}}                               
\newcommand{\Pas}{\text{$\prbt$-}\as}                       

\def\eqsepv{\; , \enspace}                                  
\def\eqfinv{\; ,}                                           
\def\eqfinp{\; .}                                           

\newcommand{\Uad}{\partie{U}\ad}

\ifdefined\proof\else \newenvironment{proof}{\small{\bf Proof.}}{\hfill$\Box$\normalsize\bigskip} \fi
\newcommand{\finpreuvesymb}{$\Box$}
\newcommand{\finremarksymb}{$\Diamond$}
\newcommand{\finexemplesymb}{$\triangle$}
\newcommand{\finpreuve}{\ \hspace*{\fill}\finpreuvesymb}
\newcommand{\finremark}{\ \hspace*{\fill}\finremarksymb}
\newcommand{\finexemple}{\ \hspace*{\fill}\finexemplesymb}

\makeatletter
\ifdefined\endproof\else \def\endproof{\finpreuve\@endtheorem}\fi
\def\endremark{\finremark\@endtheorem}
\def\endexample{\finexemple\@endtheorem}
\makeatother

\newcommand{\gph}{\mathop{\mathrm{gph}}}

\newcommand{\cl}{\mathop{\mathrm{cl}}}

\newcommand{\borel}[1]{\tribu{B}({#1})}

\newcommand{\expec}[1]{\espe\left(#1\right)}

\newcommand{\add}{^{\Sigma}}
\newcommand{\coupl}{^{C}}

\newcommand{\normdual}[1]{\left\|#1\right\|_{\star}}